\documentclass[artpaper,11pt]{article}%[article, 11pt]{amsart}%[artpaper,11pt]{article}

\usepackage{hyperref} % Créer des liens et des signets

\usepackage{cite}
\usepackage[english]{babel}
\usepackage{amsfonts}
\usepackage[T1]{fontenc}
\usepackage[utf8]{inputenc}  
\usepackage{graphicx}  
\usepackage{amsmath}
\usepackage{array}
\usepackage{enumerate}                                                                                 
\usepackage{amsthm}
\usepackage{mathabx}
\usepackage{geometry} 
\usepackage{pgf,tikz}
\usepackage{xcolor}
\usepackage{mathrsfs}
\usepackage{amsfonts}
\usepackage{amssymb}
\usepackage{amstext}
\usetikzlibrary{arrows}
\usetikzlibrary{arrows}

\DeclareMathOperator{\R}{\mathbb{R}}

\DeclareMathOperator{\uu}{\mathbf{u}}
\DeclareMathOperator{\vv}{\mathbf{v}}
\newtheorem{thm}{Theorem}[section] % reset theorem numbering for each chapter

\newtheorem{theo}[thm]{Theorem}

\newtheorem{rema}[thm]{Remark}
\newtheorem{lemm}[thm]{Lemma}

\newtheorem{hypo}[thm]{Assumptions}

\def\cms#1{\marginpar{\raggedleft\tiny{\textcolor{red}{Benoîte} : \textcolor{purple}{#1}}}}
\begin{document}
\title{Long-time behavior and darwinian optimality for an asymmetric size-structured branching process}
\author{Bertrand Cloez\footnote{MISTEA, Univ Montpellier, INRAE, Institut Agro, Montpellier, France}, Tristan Roget\footnote{IMAG, Univ Montpellier, CNRS, Montpellier, France and École Polytechnique, IP Paris, Palaiseau, France}, Benoîte de Saporta\footnote{IMAG, Univ Montpellier, CNRS, Montpellier, France}}
\maketitle

\begin{abstract}
We study the long time behavior of an asymmetric size-structured measure-valued growth-fragmentation branching process that models the dynamics of a population of cells taking into account physiological and morphological asymmetry at division. We show that the process exhibits a Malthusian behavior; that is that the global population size grows exponentially fast and that the trait distribution of individuals converges to some stable distribution. The proof is based on a generalization of Lyapunov function techniques for non-conservative semi-groups. We then investigate the fluctuations of the growth rate with respect to the parameters guiding asymmetry. In particular, we exhibit that, under some special assumptions, asymmetric division is optimal in a Darwinian sense.
\end{abstract}

\tableofcontents

\newpage

%%%%%%%%%%%%%%%%%%%%%%%%%%%%%%%%
\section{Introduction}
%%%%%%%%%%%%%%%%%%%%%%%%%%%%%%%%
%\cmbc{Motivation bio}
The aim of this paper is to study the long-time behavior of an asymmetric size-structured growth-fragmentation branching process for population dynamics.
This work is motivated by recent biological experiments \cite{proenca2018age,stewart2005aging,wang2010robust} that detected asymmetry in cell division for the species \textit{Escherichia coli}. \textit{E. coli} is a rod shaped bacterium that grows exponentially with some elongation rate then divides roughly in the middle into two genetically identical daughter cells. Each daughter cell therefore creates a new pole at division and inherits the oth er pole from its mother. After two divisions, it is possible to distinguish sister cells: one has inherited the old pole of its mother while the other one has inherited the new pole of its mother. The former is called the \emph{old pole cell}, and the latter the \emph{new pole cell}. %When it divides, one daugther inherits of the old pole (the old daughter) and the other inherits of the new pole (the young daughter). 
It is possible to track experimentally the status (old pole or new pole) of each cell together with their sizes along time and lineages, see  \cite{proenca2018age,stewart2005aging,wang2010robust}. These experiments showed that there is a statistically significant difference between the elongation rates of the old pole and new pole cells \cite{CSBIGS18}. This is is called physiological asymmetry throughout this paper. There is also a statistically significant difference between the sizes at birth of the old pole and new pole cells.% \cite{proenca2018age}\cmbc{Il y a quoi dans ce papier? $\theta_0>\theta_1$ je croyais que c'était pa remarqué}. 
This phenomenon is called morphological asymmetry.
To date, the biological mechanisms leading to these behavioral differences are not yet understood.
The aim of this paper is to propose a model for the dynamics of a population of cells taking into account both physiological and morphological asymmetry, and to compare its theoretical properties to that of the symmetric model. In particular, we study how these asymmetric properties influence the growth speed of the population.

%We observe that the cell with the old pole has a lower elongation rate (see \cite{stewart2005aging},\cite{proenca2018age}). In this note, we investigate the link between this asymmetry and a morphological asymmetry during the division. Indeed, E. Coli is well-known to divide symmetricaly with a small variance \cite{gupta2014robustness}. 

%\cmbc{Presentation informelle du model}
Let us introduce informally our model. %; more details will be given in Section~\ref{subse:definition}. 
We consider a cell population where every individual is represented by two traits $(x,p)$, where $x$ is its size and $p\in \{0,1\}$ is its status, typically $0$ for the old pole cell and $1$ for the new pole cell. These traits and the number of individuals in the population evolve randomly in continuous time as follows:
\begin{itemize}
\item each individual divides (\textit{i.e.} dies and gives birth to two new individuals) independently from the others (conditionally to the past) following an exponential clock with size-dependent intensity $B$;
\item between divisions, the size of an individual of trait $(x,p)$ grows exponentially with status-dependent elongation rate $\alpha_p$ and its status remains constant;
\item at division, an individual of trait $(x,p)$ dies and gives birth to two individuals of trait $(\theta_0x,0)$ and $(\theta_1x,1)$, with $\theta_0+\theta_1=1$.
\end{itemize}%\cmbc{Il y a peu de choses à faire pour prendre des $\theta_0$ aléatoire, on le fait?}\cms{Si on prend des $\theta$ aléatoires, il est logique de mettre des $\alpha$ aléatoire aussi, non ? Je verrai plutôt une remarque en conclusion disant que ça marcherait  pareil (?) avec des $\theta$ aléatoires, plutôt que de tout alourdir... }
The parameters $\alpha_0$ and $\alpha_1$ represent the possibly different elongation rates of old pole and new pole cells and models physiological asymmetry. The parameters $\theta_0$ and $\theta_1$ represent the proportion of the size of the mother inherited by each sister cell, thus taking into account morphological asymmetry. These parameters will not be assumed to be random %(which seems realistic in view of the data [ref?]) 
even if our results can be easily generalized to this hypothesis. To retrieve a fully symmetric model, simply take $\alpha_0=\alpha_1$ and $\theta_0=\theta_1=1/2$.

This model belongs to the class of growth-fragmentation dynamics that has attracted a lot of attention in the litterature, see for instance \cite{michel2006existence,michel2006optimal,olivier2016does,bertoin2018probabilistic,campillo2017,doumic2015statistical,guillemin2004aimd,C17, bansaye2019nonconservative} and references therein, among many others. However, it is more general than classical size-structured growth-fragmentation models as, to our best knowledge, it is the first model to take physiological and morphological asymmetry into account.%\cms{vérifiez si ce que je dis est vrai. En mettre un peu plus sur l'état de l'art des growth-fragmentation ?}

%\cmbc{Presentation informelle des résultats principaux}
As may be expected in the classical study of growth-fragmentation models, we begin by demonstrating that our model exhibits Malthusian behavior. This result    %presented in Theorem~\ref{th:mainvp},
 states that the population size grows exponentially fast and that the trait distribution converges to some stable distribution. The exponential growth rate $\lambda$ is the eigenvalue of some non-local and non-diffusive operator. Although existence and uniqueness of these eigenelements are expected for such branching models, there is no simple and systematic method to prove it. Proving such results for related growth-fragmentation models is tricky and has attracted a large amount of research in recent years \cite{bertoin2018probabilistic,bertoin2019feynman,doumic2010eigenelements,balague2012fine,caceres2011rate,mischler2016spectral,marguet2019law, bansaye2019nonconservative}. Most of the techniques used in these works cannot be applied to our problem. For instance they require regularity conditions on the operators imposing in particular that the distribution of $(\theta_0,\theta_1)$ cannot be deterministic. Instead, we use the approach of \cite{bansaye2019nonconservative, CG19} based on irreducibility and Lyapunov functions. 
%
%Again, although this type of result is expected, 
It is not straightforward that the Malthusian behavior holds for deterministic values of $(\theta_0,\theta_1)$. %Indeed, % rit is nevertheless surprising when $\theta_0,\theta_1$ are deterministic. 
Indeed, in this setting, the physiological symmetric model ($i.e$ for elongation rates) does not exhibit a Malthusian behavior. The population grows exponentially fast but the size distribution does not stabilize. It is shown in \cite{bernard2016cyclic,gabriel2019periodic} that the latter oscillates at frequencies that depend on the initial configuration. It is because of this atypical property that our demonstration of Malthusian behaviour is delicate. Asymmetry of the elongation rate therefore guarantees the vanishing of the initial condition as well as the absence of oscillation. That is an interesting first conclusion of our results from a biological point of view.
As we aim to study the influence of asymmetry, we also study the variability of the growth rate of the population with respect to the  variation of $\alpha_1-\alpha_0$ and $\theta_1-\theta_0$. From an evolutionary point of view \cite{metz1996adaptive,metz2006fitness,dieckmann1996dynamical,geritz1998evolutionarily}, the Malthusian rate $\lambda$ is called the \textit{fitness} and determine if a mutant population can invade a resident one: a mutant with a larger fitness should invade the resident population. We compute the partial derivatives of $(\alpha_1,\alpha_0,,\theta_1,\theta_0) \mapsto \lambda = \lambda(\alpha_1,\alpha_0,,\theta_1,\theta_0)$ similarly as in \cite{campillo2017,gaubert2015discrete,michel2006optimal,olivier2016does} for related models. As in these works, these formulas involve unknown quantities such as eigenvectors. To overcome this problem, numerical simulations are often used. Instead, we focus on the particular case $B(x)=x$ which includes the idea that large cells divide faster than small cells. For this special division rate, we establish new formulas for the asymptotic distribution, even in the symmetric case. We then extend some results of \cite{zaidi2015,hall1989functional, hall1990functional} which establish some explicit formulas for the asymptotic distribution. To derive them, we show that a clever transformation of the quantities involved satisfies a functional equation with known solutions. As a consequence, we show that asymmetry is optimal in a Darwinian sense. That is an interesting second conclusion of our results from a biological point of view. 

This paper is organized as follows.
In Section \ref{se:mainresult}, we define the measure-valued branching process modeling the physiologically and morphologically asymmetric cell division, and we make the connection with semigroup theory and partial differential equations.
In Section \ref{se:Malthus}, we prove our first main result concerning the long-time behavior of the measure-valued branching process.
In Section \ref{se:darwin}, we study the sensitivity of the Malthusian parameter as a function of the parameters of the model, we prove our second main result  and we establish several explicit formulas in the special case $B(x)=x$. 
%
%%%%%%%%%%%%%%%%%%%%%%%%%%%%%%%%
\section{Definition of the model and main results}
\label{se:mainresult}
%%%%%%%%%%%%%%%%%%%%%%%%%%%%%%%%
%
In this section, we precisely define our asymmetric size-structured branching process and state our main results regarding its asymptotic behavior: existence of eigenelements, which can be interpreted biologically as the Malthusian behavior and sensibility analysis of these eigenelements with respect to the assymetry parameters. In the special case where the division rate is the identity function, we state in addition the Darwinian optimality of the asymmetric model. 
%
%%%%%%%%%%%%%%%%%%%%%%%%%%%%%%%%
\subsection{Asymmetric branching process}
\label{subse:definition}
%%%%%%%%%%%%%%%%%%%%%%%%%%%%%%%%
%
In this section, we define the measure-valued branching process we use to model physiologically and morphologically asymmetric cell division. It can be seen either as a branching process \cite{harris1964theory}, a piecewise deterministic Markov process \cite{Davis93} or a stochastic differential equation with jumps \cite{ikeda2014stochastic}. Therefore we detail the model in these three frameworks. We also introduce here all our notation and explain the link of our model with the partial differential equations theory.
%
%%%%%%%%%%%%%%%%%%%%%%%%%%%%%%%%
\paragraph{Branching process path-wise construction}
%%%%%%%%%%%%%%%%%%%%%%%%%%%%%%%%
%We give the path-wise explicit construction, through a path construction, a PDMP definition and a stochastic differential equation with jump. 
%
Throughout the paper,  we use of the classical Ulam-Harris-Neveu notation~\cite{dawson1993measure} to identify each individual in a genealogical tree. Let 
\begin{equation*}
\mathcal{U}=\bigcup_{n\in\mathbb{N}}\left\{0,1\right\}^n,
\end{equation*}
denote the set of all the descendants of the original (unique) ancestor.
The original ancestor is labeled by $\emptyset$ and is identified to $\{0,1\}^0$. When an individual $u\in\mathcal{U}$ dies (divides), it gives birth to two descendants labelled $u0$, $u1$. We denote by $b_u$ and $d_u$ the birth and the death dates of individual $u$. Let $\mathcal{V}_t$ be the set of alive individuals at time $t\geq 0$; that is
$$
\mathcal{V}_t= \left\{ u\in \mathcal{U} \ | \ b_u \leq t <d_u  \right\}.
$$ %\cms{formellement l'ancêtre d'origine n'est pas compté ?}
We denote its cardinal by $N_t$; this represents the number of alive individuals at time $t$.
Every individual $u\in \mathcal{U}$ at time $t\in [b_u, d_u)$ possesses a trait $Y^u_t =(X^u_t,P^u_t)$, where $X^u_t\in \mathbb{R}_+$ is the size of individual $u$ at time $t$ and  $P^u_t \in \{0,1\}$ is its status and encodes that  individuals may have two different dynamics. The global population is described through the punctual measure
$$
Z_t= \sum_{u \in \mathcal{V}_t} \delta_{Y^u_t}\in \mathcal{M}^{+}_P(\R_+\times \lbrace 0,1\rbrace),
$$
%\cms{Il serait bon d'introduire une notation pour l'espace dans lequel vit Z}
where $\mathcal{M}^{+}_P(\R_+\times \lbrace 0,1\rbrace)$ denotes the set of positive and finite punctual measures on $\R_+\times\lbrace 0,1\rbrace$.\\

Let us now describe the population random dynamics. For all $t\geq 0$, we set $\mathcal{F}_t=\sigma\{V_s, (Y^u_s)_{u\in V_s}, \forall s\leq t \}$, the $\sigma$-field generated by the traits of all individuals born before time $t$ and up to time $t$ (or their death, whichever comes first). Conditionnaly on $\mathcal{F}_t$, we have that
\begin{itemize}
\item For all $u\in \mathcal{V}_t$ and $t\leq r< d_u$, the size of individual $u$ up to its death grows exponentially with a status-dependent growth rate: we have%\cmbc{harmiser si $s$ est le temps ou l'incrément de temps}
$$
X^u_r= X^u_t \exp({\alpha_{P_t^u} (r-t)}), \qquad P^u_r=P^u_t,
$$
where $\alpha_0>0$ and $\alpha_1>0$ are the (possibly) different elongation rates modeling physiological asymmetry. The status is constant until death.
\item For all $u\in \mathcal{V}_t$, the death dates $d_u$ are independent random variables with distribution given by
$$
\mathbb{P}(d_u > t+s \ | \ \mathcal{F}_t) = \exp\left(-\int_0^s B(X^u_t \exp({\alpha_{P_t^u} r})) dr \right),
$$ 
where $B$ is a measurable function from $\mathbb{R}_+$ onto $\mathbb{R}_+$ representing the size-dependent division rate. Indeed, as stated above $X^u_t \exp({\alpha_{P_t^u} r})$ is the size at time $t+r$ of individual $u$ given that it did not die between dates $t$ and $r$.%\cms{vérifier si ma def de B est bonne}
\item For all $u \in \mathcal{V}_t$, we have $b_{u0}=b_{u1}=d_u$, meaning that individual $u$ dies (divides) and at the same time gives birth to two individuals $u0$ and $u1$, and for $i\in\{0,1\}$,
$$
Y^{ui}_{b_{ui}} = (\theta_i X_{d_u -}^u, i)= (\theta_i X_{b_u}^u e^{\alpha_{P_t^u} (d_u-b_u)},i).
$$
This equation means that the two new individuals get a different status, individual $u0$ with status $0$ inherits a proportion $\theta_0$ of the size at death of individual $u$, and individual $u1$ with status $1$ inherits a proportion $\theta_1$ of the size at death of individual $u$. One has $\theta_0>0$, $\theta_1>0$ and $\theta_0+\theta_1=1$. Allowing $\theta_0$ and $\theta_1$ to differ from $1/2$ models morphological asymmetry.
\end{itemize} 
This model is well defined until the explosion time $T$ such that $N_T=+\infty$. We will show in Lemma~\ref{lem:existence} below, that $T=\infty$ when the division rate $B$ is locally bounded.\\ %as soon as \tcr{$B\leq C(1+x)$}.

The process $(Z_t)_{t\geq 0}$ belongs to the class of measure-valued piecewise deterministic Markov processes introduced in \cite{CdJ} and satisfy a stochastic differential equations with jumps, as detailed below. In the sequel, we denote $\mathbb{E}_{(x,p)}$ and $\mathbb{P}_{(x,p)}$ respectively the expectation and probability conditionally to $Z_0= \delta_{(x,p)}$. For any measurable function $f:\mathbb{R}_+ \times \mathbb{R}_+ \times \{0,1\} \to \mathbb{R}$, set 
$$Z_t(f) =Z_t(f_t) 
= \sum_{u \in \mathcal{V}_t} f(t,Y^u_t)= \sum_{u \in \mathcal{V}_t} f_t(Y^u_t).$$
%
%%%%%%%%%%%%%%%%%%%%%%%%%%%%%%%%
\paragraph{Piecewise deterministic Markov process framework}
%%%%%%%%%%%%%%%%%%%%%%%%%%%%%%%%
The only source of randomness of the process comes from the division clocks $d_u$. The special form of the distribution of the division clocks yields that the measure-valued process $(Z_t)_{t\geq 0}$ is a measure-valued piecewise deterministic Markov process. Its local characteristics, as defined in \cite{CdJ}, are as follows. For any punctual measure  $\zeta=\sum_{j=1}^{n}\delta_{(x_j,p_j)}\in\mathcal{M}^{+}_P(\R_+\times\lbrace 0,1\rbrace)$, and $t\geq 0$,%\cms{donner un nom à l'espace d'états}
\begin{itemize}
\item[•] the flow is defined by
\begin{equation*}
\Phi(\zeta,t)=\sum_{j=1}^{n}\delta_{(x_j e^{\alpha_{p_j} t},p_j)},
\end{equation*}
\item[•] the jump intensity is $\lambda(\zeta)=\sum_{j=1}^{n}B(x_j)$,
\item[•]the Markov jump kernel is given by
\begin{equation*}
\mathcal{Q}(\zeta,A)=\sum_{k=1}^{n}\frac{B(x_k)}{\sum_{j=1}^{n}B(x_j)} \mathbf{1}_{A}(\zeta -\delta_{(x_k,p_k)}+\delta_{(\theta_0x_k,0)}+\delta_{(\theta_1 x_k, 1)}),
\end{equation*}
for all Borel subset $A$ of $\mathcal{M}^{+}_P(\R_+\times \lbrace 0,1\rbrace)$.
\end{itemize}
%
%%%%%%%%%%%%%%%%%%%%%%%%%%%%%%%%
\paragraph{Stochastic differential equation framework}
%%%%%%%%%%%%%%%%%%%%%%%%%%%%%%%%
 The dynamics of the measure-valued process $(Z_t)$ can also be described in terms of stochastic differential equation with jumps. Let $\mathcal{N}(ds, du, dl)$ be a Poisson point measure on $\mathbb{R}_+\times \mathcal{U} \times \mathbb{R}_+$ of intensity $ds\, n(du)\, dl$ where $ds, dl$ are Lebesgue measures and $n(du)$ the counting measure on $\mathcal{U}$. If $f:\mathbb{R}_+ \times \mathbb{R}_+ \times \{0,1\} \to \mathbb{R}$ is a bounded measurable function with bounded measurable derivatives then denote, one has
\begin{align}
Z_t(f_t) 
&= \sum_{u \in \mathcal{V}_t} f_t(Y^u_t) = \sum_{u \in \mathcal{V}_t} f_t (X^u_t, P^u_t)\label{eq:SDE}\\
&= \sum_{u\in \mathcal V_0} f_0 (Y^u_0) + \int_0^t \sum_{u\in \mathcal V_0} \left( \partial_s f_s( Y^u_s ) +  \alpha_{P_s^u} X^u_s \partial_x f_s\left( Y^u_s \right) \right)ds \nonumber \\
&\quad +  \int_{[0,t] \times \mathcal{U} \times \mathbb{R}_+} \hspace*{-0.2cm} \mathbf{1}_{u \in \mathcal V_s, l \leq B(X^u_{s-}) } \left( f_s( \theta_0 X^u_{s-}, 0) + f_s( \theta_1 X^u_{s-}, 1) -f_s( Y^u_{s-}) \right) \mathcal{N}(ds,du, dl).\nonumber
%&= Z_0(f_0) + \int_0^t \partial_s f_s(
\end{align}
See \cite{VTC,FM04,BM15, M19} for details. 
%
%%%%%%%%%%%%%%%%%%%%%%%%%%%%%%%%
\paragraph{Transitions semi-group}
%%%%%%%%%%%%%%%%%%%%%%%%%%%%%%%%
We can naturally associate to $(Z_t)_{t\geq 0}$ the semigroup $(M_t)_{t\geq 0}$ defined for any non-negative measurable function $f:\R_+\times \{0,1\}\rightarrow \R$ by 
\begin{equation}
\label{eq:semigpe}
M_t f(x,p)=\mathbb{E}\left[\sum_{u \in \mathcal{V}_t} f(Y^u_t) \ \big\vert \ Z_0=\delta_{(x,p)}\right] = \mathbb{E}_{(x,p)}\left[\sum_{u \in \mathcal{V}_t} f(Y^u_t) \right],
\end{equation}
which describes the mean behavior of $Z_t(f)$. We will see in Lemma~\ref{lem:existence}, that $M_t$ also acts on bounded functions $f$ through Equation~\eqref{eq:semigpe}. Let us define $C_+^1((0,+\infty)\times\lbrace 0,1\rbrace)$ the space of non negative and continuous functions with continuous derivative with respect to the variable $x\in\R_+$. %\cmbc{A voir si on definit cette ensemble, et si oui plus tôt juste au dessus. J'ai changé le est definie pour les fonctions par se lit sur les fonction par qui n'exclut pas qu'on peut lui appliquer autrechose}  
In Lemma~\ref{lem:generateur}, we show that the extended generator $\mathcal{A}$ of
$(M_t)_{t\geq 0}$ reads
\begin{equation*}
\mathcal{A}f(x,p)=
\alpha_p x \partial_x f(x,p) + B(x)\left( f(\theta_0 x, 0) +f(\theta_1 x,1)- f(x,p)\right),
\end{equation*}
for all $f\in C_+^1((0,+\infty)\times\lbrace 0,1\rbrace)$.
The dual semigroup $M^{*}_t\mu = \mu M_t$ describes the mean behavior of the process $Z$, that is $\mu M_t =\mathbb{E}\left[Z_t\vert Z_0=\mu\right]$, for $\mu\in \mathcal{M}^{+}_P(\R_+\times \lbrace 0,1\rbrace)$. Let us now end this subsection by a link with partial differential equation theory. If we define $(\mu_t),(\mu^0_t)$ and $(\mu^1_t)$ by the equality
$$
\mu M_t=\mu_t(dx,dp)=\delta_{0}(dp)\mu^0_t(dx) + \delta_{1}(dp)\mu^1_t(dx),
$$
we obtain the following system of growth-fragmentation equations satisfied by $(\mu^0,\mu^1)$: for all $ p \in\lbrace 0,1\rbrace$, one has

\begin{align}\label{eq:general}
\partial_t \mu_t^p(x) & +\partial_x(\alpha_p x \mu_t^p(x))+B(x)\mu_t^p(x)\\
&= \frac{1}{\theta_p} B\left(\frac{x}{\theta_p}\right)\mu_t^0\left(\frac{x}{\theta_p}\right)+\frac{1}{\theta_p} B\left(\frac{x}{\theta_p}\right)\mu_t^1\left(\frac{x}{\theta_p}\right) \nonumber.
\end{align} 
Equation (\ref{eq:general}) is a system of growth-fragmentation equations with growth rates variability. To our knowledge, this equation (or more precisely, this system of equations) has never been introduced before. Our main result (see Theorem \ref{th:mainvp}) states that the solutions of Equation (\ref{eq:general}) converges at exponential speed to some stable distribution in some weighted $L^1-$norm.%
%%%%%%%%%%%%%%%%%%%%%%%%%%%%%%%%
\subsection{Existence of eigenelements and spectral gap inequality}
\label{subse:mainresultgeneral}
%%%%%%%%%%%%%%%%%%%%%%%%%%%%%%%%
%

In this section we state our main results concerning the eigenelements of the general asymmetric model.
We first make two assumptions to avoid atypical behaviors. 
\begin{hypo}\label{hyp:B}\quad
\begin{itemize}
\item[(i)] Function $B$ is a positive and continuous function on $(0,+\infty)$ such that
\begin{equation*}
\lim_{x\to 0} B(x) = 0, \quad \lim_{x\to \infty} B(x) = + \infty.
\end{equation*}
\item[(ii)] Elongation rates verify $\alpha_0\neq \alpha_1$.
\end{itemize}
\end{hypo}
The first assumption states that small cells do not divide and very large cells divide at once. To see informally the necessity of this type of assumptions, consider the simpler process $(X_t)_{t\geq 0}$ modeling a single cell lineage (without branching) with a constant division rate ($i.e.$ jump rate) $B$. This process increases exponentially between jumps and jumps from $X_{T-}$ to $X_T=\Theta X_{T-}$ at Poissonian times $T$; where $\Theta$ is a random variable taking value in $\{\theta_0, \theta_1\}$. This process is then the exponential of a Levy process. Thus, it has 3 possible asymptotic behaviors: convergence to infinity, convergence to $0$ or oscillation without convergence \cite[Corollary 2 p.190]{bertoin1996levy}. To avoid such trivial behavior, we  assume that small cells do not divide and large cells divide faster. 

The second assumption is the physiological asymmetry assumption which is necessary to avoid oscillation of the size distribution depending on the initial state as shown in \cite{bernard2016cyclic,gabriel2019periodic}.

Our first significant result concerns the existence of eigenelements and the convergence of the semigroup at exponential speed.
\begin{theo}
\label{th:mainvp}
Suppose Assumptions~\ref{hyp:B} hold. Then there exist a probability measure $\gamma$, a measurable function $h:(0,+\infty)\times\lbrace 0,1\rbrace\to (0,+\infty)$ and $\lambda>0$ such that $\gamma(h)=1$ and
\begin{equation}\label{eq:eigen}
\forall t\geq 0, \qquad M_t h =e^{\lambda t} h, \qquad \gamma M_t=M^*_t\gamma=e^{\lambda t} \gamma.
\end{equation}
Moreover, $h\leq V$, $\gamma(V)$ is finite and there exist $C,\omega>0$ such that for all $t\geq 0$ and measure $\mu$, one has\
\begin{equation}
\label{eq:cvexpo}
\sup_{\| f/V\|_\infty \leq 1} \left|e^{- \lambda t} \mu M_tf - \mu(h) \int f d\gamma \right| \leq C e^{-\omega t} \mu(1+V),
\end{equation}
where $V: x\mapsto x^q +\frac{1}{x^q}$ for some $q>0$, and the supremum in \eqref{eq:cvexpo}  is taken over all measurable functions $f$ such that $f/V$ is bounded by $1$.%\cmbc{good? J'ai virer l'unicité qu'on voit à cause de la limite}
\end{theo}

In other words, this theorem reads %\cms{préciser le sens du $O$: en temps ? + uniformité en $f$ ?}
$$
\mathbb{E}_{(x,i)} \left[ \sum_{u \in \mathcal{V}_t} f(X^u_t, P^u_t) \right] = h(x,i) e^{\lambda t} \gamma(f) + O(e^{(\lambda- \omega)t}),
$$
where $O$ is the Landau notation; this mean that $e^{-(\lambda- \omega)t} O(e^{(\lambda- \omega)t})$ is uniformly bounded over measurable functions $f$ such that $f/V$ is bounded.

Setting $f=1$, we see that the mean number of individuals grows exponentially at rate $\lambda$, which is called Malthusian behavior in population dynamics. Parameter $\lambda$ is called the Malthusian parameter. In addition we can prove that $\lambda$ is between $\alpha_0$ and $\alpha_1$; see Remark~\ref{rq:borne-lambda}.

Inequality~\eqref{eq:cvexpo} ensures the uniqueness (up to multiplicative constants) of the eigenelements. More precisely,  if there exists a measurable function $\widetilde{h}$, bounded by $V$, and a number $\widetilde{\lambda}$ such that for all $t\geq 0$ (or at least one), $M_t \widetilde{h} = e^{\widetilde{\lambda} t} \widetilde{h}$ then by choosing $\mu=\delta_x$ and  $f=\widetilde{h}$ in Equation~\eqref{eq:cvexpo} then we see that $\widetilde{\lambda}=\lambda$ and $\widetilde{h}=\gamma(\widetilde{h}) h$. Similarly, probability measure $\gamma$ is unique.

When $\alpha_0 = \alpha_1$, the existence and uniqueness of a unique triplet $(\lambda, h,\gamma)$ of eigenelements that satisfies (\ref{eq:eigen}) was proven in \cite{doumic2010eigenelements} for the symmetric equation (with one cell population) but the convergence (\ref{eq:cvexpo}) is false \cite{bernard2016cyclic}. Using this result, we prove in Lemma~\ref{lem:Utogamma} and Lemma~\ref{lem:hunique} that existence and uniqueness also hold true for our system of two cell equations.

We can go further than exhibiting the mean behavior of the process. Indeed, we can study the variation of the Malthusian parameter as a function of the parameters of the model. To do so, we introduce the following change of variable. Let $\alpha_1=\alpha + \epsilon$, $\alpha_0= \alpha -\epsilon$ in a such way that $\alpha = (\alpha_0 +\alpha_1)/2$ and $\epsilon = (\alpha_1 -\alpha_0)/2$, $\theta=\theta_0$ (recall that $\theta_1=1-\theta_0$) and $\uu=(\alpha,\epsilon,\theta)\in (0, +\infty)^2 \times (0,1)$.

We now study the eigenelements $(\lambda,\gamma,h)$ as functions of $\mathbf{u}$. However, for the sake of simplicity we do not highlight this dependence on the parameter $\mathbf{u}$ in the notation (it will be done in Section~\ref{se:darwin}).

\iffalse

We denote now by $\mathcal{A}_u$ instead of $\mathcal{A}$ the generator of the semigroup $(M_t)_{t\geq 0}$, namely we highlight the dependence on the parameters. We do the same for the eigenelements $(\lambda_u,h_u,\gamma_u)$ :
\begin{align*}
&\mathcal{A}_u h_u = \lambda_u h_u, \quad \mathcal{A}^{*}_u \gamma_u =\lambda_u \gamma_u, \quad \int h_u \gamma_u =1,\quad \int \gamma_u = 1. 
\end{align*}
We are interested in the variation of the principal eigenvalue $u\mapsto \lambda_u$. For all $u=(\alpha,\epsilon,\theta)\in (0,+\infty)^2\times (0,1)$ we note $\lambda(u)=\lambda_u$.\cms{plein de notations pas définies ici, à reprendre}
\fi

\begin{theo}\label{pr:variationmalthus}
Under Assumptions~\ref{hyp:B} (i) and if $B$ is $C^1$ then $h$ is $C^2$ and
\begin{itemize}
\item[(i)]
\begin{equation*}
\frac{\partial\lambda}{\partial \alpha} =\int_{\lbrace 0,1\rbrace}\int_{0}^{+\infty}x \partial_x h(x,p)\gamma(dx,dp),
\end{equation*}
\item[(ii)]
\begin{equation*}
\frac{\partial\lambda}{\partial \epsilon} =\int_{\lbrace 0,1\rbrace}\int_{0}^{+\infty} (2p-1) x \partial_x h(x,p)\gamma(dx,dp),
\end{equation*}
\item[(iii)]
\begin{equation*}
\frac{\partial \lambda}{\partial \theta} = \int_{\lbrace 0,1\rbrace}\int_{0}^{+\infty}B(x)\left[\partial_xh(\theta x,0)-\partial_x h((1-\theta)x,1)\right]\gamma(dx,dp).
\end{equation*}
\end{itemize}
\end{theo}

On the one hand, this result establishes the regularity of the Malthusian parameter. On the other hand, it extends some results of \cite{michel2006existence,michel2006optimal} to our asymmetric framework. In the symmetric case $\alpha_0 =\alpha_1=\alpha$, the eigenvalue is $\lambda=\alpha$ and the eigenfunction is $h:(x,p)\mapsto x$. From Theorem~\ref{pr:variationmalthus} $(ii)$ the influence of physiological asymmetry is related to the asymptotic mean size of the cells. Unfortunately, the asymptotic measure $\gamma$ is generally  unknown. From Theorem \ref{pr:variationmalthus} $(i)$, the malthusian parameter $\lambda$ is increasing with $\alpha$. There are no similar arguments in the non-symetric case.
 %When the eigenmeasure is explicit in the symmetric model, as in \cite{hall1989functional,hall1990functional,zaidi2015}, we can express $\gamma$ (which also supplies the poles\cms{je ne comprends pas cette phrase}) though Lemma~\ref{lem:Utogamma} below.
%
%%%%%%%%%%%%%%%%%%%%%%%%%%%%%%%%
\subsection{The particular case $B(x)=x$}
\label{subse:particular}
%%%%%%%%%%%%%%%%%%%%%%%%%%%%%%%%
%
In the special case where the division rate $B$ equals the identity function, we obtain more explicit results such as the shape of $\gamma$, the moments of $\gamma$, etc which generalize results of \cite{hall1989functional,hall1990functional}. These results are in Section~\ref{se:darwin}. Together with Theorem~\ref{pr:variationmalthus}, these additional properties yield the Darwinian optimality of asymmetry which reads as follows.

\begin{theo}\label{th:inegpartial}
 Let $\theta\in (0,1)$ be such that $1-\theta < \theta$ and $\alpha\in(0, +\infty)$. At $\uu= (\alpha,0,\theta)$, we have
\begin{equation*}
\frac{\partial \lambda}{\partial \epsilon} < 0.
\end{equation*}
\end{theo}

Theorem~\ref{th:inegpartial} implies that if a cell divides into two daughter cells with morphological asymmetry ($\theta\neq 1/2$), then it is optimal, in the Darwinian sense that the Malthusian parameter is increased, that the two daughter cells have different elongation rates and thus also exhibit physiological assymetry. More precisely,  it is advantageous for the largest cell at division %($1-\theta < \theta$ yields $\theta=\theta_0>1/2$) 
to elongate faster. \\%($\elpsilon<0$ yields $\alpha_0>\alpha_1$).

The rest of this paper is dedicated to the proofs of our main results theorems \ref{th:mainvp}, \ref{pr:variationmalthus}, \ref{th:inegpartial} and some additional results.

%%%%%%%%%%%%%%%%%%%%%%%%%%%%%%%%
\section{Malthusian behavior: eigenelements of the semi-group}
\label{se:Malthus}
%%%%%%%%%%%%%%%%%%%%%%%%%%%%%%%%
%
This section is dedicated to the proof of Theroem \ref{th:mainvp} exhibiting the eigenelements of the semi-group of our asymmetric branching process. Biologically speaking, it establishes the Malthusian behavior of the asymmetric model. We start with preliminary results concerning the non-explosion of the process and its infinitesimal generator in section \ref{sec:explo}, and then proceed to the proof in section \ref{sec:proof22} by using the approach developed in \cite{bansaye2019nonconservative}. %\cms{remettre une couche ici pour redire pourquoi on ne peut pas directement appliquer des résultats existants ?}
%
%%%%%%%%%%%%%%%%%%%%%%%%%%%%%%%%
\subsection{Non explosion and martingale properties}%infinitesimal generator}%{Proofs of Section \ref{subse:definition}}
\label{sec:explo}
%%%%%%%%%%%%%%%%%%%%%%%%%%%%%%%%
%
We first establish that under mild condition on the division rate $B$, the process does not explode in finite time. To that end, we introduce the notation:
$$
\overline{\alpha}= \max(\alpha_0,\alpha_1), \qquad \underline{\alpha}= \min(\alpha_0,\alpha_1).
$$
\begin{lemm}
\label{lem:existence}
If $B$ is locally bounded on all intervals of the type $[0, M]$, for all $M\geq 0$, then the population does not explode and for all $(x,p)\in \mathbb{R}_+\times\{0,1\}$ and $T>0$, one has
$$
%% \mathbb{E}_{(x,p)}[N_T] \leq \exp\left(T \sup_{y \leq xe^{\overline{\alpha} T}} B(y) \right).
  \mathbb{E}_{(x,p)}[N_T] \leq \exp\left( T \sup_{y \leq xe^{\overline{\alpha} T}} B(y) \right).
$$
In particular, $(M_t)_{t\geq 0}$ acts on bounded and measurable functions.
\end{lemm}

\begin{proof}%[Proof of Lemma \ref{lem:existence}]
Starting with one cell with size $x$ and status $p$, all its descendants have size lower than $xe^{\overline{\alpha} T}$ up to time $T$. Using for instance the Gillepsie algorithm, one can couple our model on $[0,T]$ with a simple Yule process $(\Upsilon_t)_{0\leq t\leq T}$ with branching rate
$$
\mathbf{B} = \sup_{y \leq xe^{\overline{\alpha} T}} B(y),
$$
in such a way that the number $N_t$ of individuals in the original process at time $t\leq T$ is bounded by $\Upsilon_t$. See for instance \cite[Section 8 p.105]{harris1964theory} for the definition and properties of Yule process. Finally as $\mathbb{E}_1[\Upsilon_t] \leq e^{\mathbf{B} t}$, one obtains
$$
 \mathbb{E}_{(x,p)}[N_T] \leq \mathbb{E}_1[\Upsilon_T] \leq e^{T\mathbf{B}}.
$$
As this quantity is finite for all $T$, the process does not explode in finite time.
\end{proof}

An alternative proof could be to use the SDE~\eqref{eq:SDE} with a stopping time argument as in~\cite[Theorem 4.1]{FM04}.\\

Let us define the operator $\mathcal{A}$ acting on the space of $C^1$ functions $f$ by
\begin{align}\label{eq:generateur}
\mathcal{A}f(x,p)=
&\alpha_p x \partial_x f(x,p) + B(x)\left( f(\theta_0 x, 0) +f(\theta_1 x,1)- f(x,p)\right),
\end{align}
for every $(x,p)\in\R_+\times\lbrace 0,1\rbrace$.

In the following lemma, we derive a Duhamel type formula (variation of constants formula for semigroups) describing the evolution of $(M_t)_{t\geq 0}$. Consequently, we show that operator $\mathcal{A}$ is, in the sense stated in this lemma, the generator of $(M_t)_{t\geq 0}$.

\begin{lemm}\label{lem:generateur}% \cmbc{A voir si on découpe en plusieurs lemmes ou si on en ecrit un avec une preuve plus longue etc.}
Assume that $B$ is locally bounded over all intervals of the type $[0, M]$, for all $M\geq 0$.
\begin{itemize}
\item[(i)]For all $(x,p)\in\R_+\times\lbrace 0,1\rbrace$, $t\geq 0$ and measurable functions $f$ such that $M_tf$ is well defined we have
\begin{align*}
M_tf(x,p)&=f(xe^{\alpha_p t}, p)e^{-\int_{0}^{t}B(xe^{\alpha_p s})ds}\\
&+\int_{0}^{t}e^{-\int_{0}^{s}B(xe^{\alpha_p s'})ds'} B(xe^{\alpha_p s}) M_{t-s}f(\theta_0 xe^{\alpha_p s}, 0) ds\\
&+\int_{0}^{t}e^{-\int_{0}^{s}B(xe^{\alpha_p s'})ds'} B(xe^{\alpha_p s}) M_{t-s}f(\theta_1 xe^{\alpha_p s}, 1) ds.
\end{align*}
\item[(ii)]  For all bounded $C^1$ functions $f$ such that $\mathcal{A}f$ is bounded, we have that 
$$
\left( Z_t(f) - Z_0(f) - \int_0^t Z_s(\mathcal{A} f) ds \right)_{t\geq 0}
$$
is a martingale.
\end{itemize}
\end{lemm}
\begin{proof}%[Proof of Lemma \ref{lem:generateur}]
$(i)$ We split the expression of $M_tf(x,p)$ depending on the ancestor individual being still alive at time $t$ or not. Using the branching property, we obtain
\begin{align*}
M_tf(x,p) 
&= \mathbb{E}_{(x,p)}\left[\sum_{u \in \mathcal{V}_t} f(Y^u_t) \mathbf{1}_{d_\emptyset> t}\right] +\mathbb{E}_{(x,p)}\left[\sum_{u \in \mathcal{V}_t} f(Y^u_t) \mathbf{1}_{d_\emptyset\leq t}\right] \\
&=  f(xe^{\alpha_p t},p) \mathbb{P}(d_\emptyset> t) \\
&\quad + \mathbb{E}_{(x,p)}\left[ \mathbb{E}_{(\theta_0 x e^{\alpha_p d_\emptyset}, 0)} \left[\sum_{u \in \mathcal V_{t-d_\emptyset}} f(Y^u_{t-d_\emptyset})\right] \mathbf{1}_{d_\emptyset\leq t}\right] \\
&\quad + \mathbb{E}_{(x,p)}\left[ \mathbb{E}_{(\theta_1 x e^{\alpha_p d_\emptyset}, 1)} \left[\sum_{u \in \mathcal V_{t-d_\emptyset}} f(Y^u_{t-d_\emptyset})\right] \mathbf{1}_{d_\emptyset\leq t}\right] \\
&=f(xe^{\alpha_p t}, p)e^{-\int_{0}^{t}B(xe^{\alpha_p s})ds}\\
&+\int_{0}^{t}e^{-\int_{0}^{s}B(xe^{\alpha_p s'})ds'} B(xe^{\alpha_p s}) M_{t-s}f(\theta_0 xe^{\alpha_p s}, 0) ds\\
&+\int_{0}^{t}e^{-\int_{0}^{s}B(xe^{\alpha_p s'})ds'} B(xe^{\alpha_p s}) M_{t-s}f(\theta_1 xe^{\alpha_p s}, 1) ds.
\end{align*}
$(ii)$ Fix a bounded $C^1$ function $f$ and $(x,p)\in\R_+\times\lbrace 0,1\rbrace$. From $(i)$, we have that $t\mapsto M_t f(x,p)$ is derivable at time $t=0$ and
$$
\frac{\partial M_t f}{\partial t}\Big|_{t=0}(x,p) =\mathcal{A} f(x,p).
$$
If $\mathcal{A}f$ is further bounded, by the semigroup (or Markov) property and the dominated convergence theorem, we have
that $t\mapsto M_t f(x,p)$ is derivable at every time $t$ and
$$
\frac{\partial M_t f}{\partial t} (x,p) =M_t \mathcal{A} f(x,p).
$$
Finally $(ii)$ is a consequence of the previous equation and Markov property (as in \cite[Proposition 1.7 p. 162]{EK09}).
\end{proof}
Note again that Lemma~\ref{lem:generateur} $(ii)$ could be proved through the SDE~\eqref{eq:SDE}.

%
%%%%%%%%%%%%%%%%%%%%%%%%%%%%%%%%
%\subsection{General Model}%{Proofs of Section \ref{subse:mainresultgeneral} }
%%%%%%%%%%%%%%%%%%%%%%%%%%%%%%%%
%
%%%%%%%%%%%%%%%%%%%%%%%%%%%%%%%%
\subsection{Proof of Theorem \ref{th:mainvp}}
\label{sec:proof22}
%%%%%%%%%%%%%%%%%%%%%%%%%%%%%%%%
%
To prove Theorem~\ref{th:mainvp}, we use the approach developped in \cite{bansaye2019nonconservative}. To do so, we have to verify that the semigroup $(M_t)_{t\geq 0}$ satisfies \cite[Assumptions \textbf{A}]{bansaye2019nonconservative} which are the existence of Lyapunov functions, a mass ratio  inequality and a Doeblin minoration condition. These three steps are described in the next three subsections. 
%
%The proof is divided in several parts : in Lemma \ref{lem:psi, V}, we will use the \cite[Proposition 2.10]{bansaye2019nonconservative} to prove that functions $V,\psi$ define below  are Lyapunov functions. We will prove the Doeblin minoration with ad hoc method in Lemma \ref{lem : doeblin} and use the approach of \cite{CG19} for the mass ratio in Lemma \ref{lem:massratio}.\cmbc{refaire l'ordre des lemmes et le blabla}
%
%%%%%%%%%%%%%%%%%%%%%%%%%%%%%%%%
\subsubsection{Lyapunov functions}
%%%%%%%%%%%%%%%%%%%%%%%%%%%%%%%%
%
\vspace{0.2cm}
Let $V,\phi:(0, +\infty)\times\lbrace 0,1\rbrace\rightarrow (0, +\infty)$ defined by 
\begin{equation*}
\phi(x,p)=x + \frac{1}{x},\quad V(x,p)=x^q +\frac{1}{x^q},
\end{equation*} 
for some $q\geq (\overline{\alpha}+2)/\underline{\alpha}>1$. %\cms{Y aurait-il un intérêt à écrire $V_q$ plutôt que $V$ ? Mettre la référence exacte de la def de $q$}
Note that we have $V\geq \phi \geq 0$, and that both $V$ and $\phi$ belong to $C^1_+((0,+\infty)\times \lbrace 0,1\rbrace)$.%\cmbc{Attention aux defs des espaces}

The aim of this subsection is to show that $V$ and a well-suited function $\psi$ introduced in Equation (\ref{def:psi}) below verify \cite[Assumption (A0) (A1) (A2)]{bansaye2019nonconservative}, that roughly speaking states that $\psi \leq V$, functions $M_t V, M_t\psi$ are locally bounded and 
\begin{align}\label{eq:A0A1A2}
M_\tau V \leq \alpha V + \theta \mathbf{1}_K \psi, \qquad M_\tau \psi \geq \beta \psi,
\end{align}
for some $\beta>\alpha$ and $\tau>0$, where $K\subset(0,+\infty)$ is a compact set. These assumptions guaranty some compactness (or tightness) property for the dynamics of the semigroup.

We start with establishing drift properties for $\phi$ and $V$, based on straightforward analytic calculations.

\begin{lemm}
\label{lem:psi, V}
Under Assumption \ref{hyp:B},  there exist $a,b,\zeta\in \mathbb{R}$ such that $b>a$ and
$$
\mathcal{A} \phi\geq b \phi  \ \text{ and } \ \mathcal{A} V  \leq a V + \zeta\phi.%\leq \xi \psi. 
$$
\end{lemm}

\begin{proof}
Let $\phi_0:(x,p) \mapsto x$ and $\phi_1:(x,p)\mapsto 1/x$. As the division rate $B$ is non negative, we have
$$
\mathcal{A}\phi_0\geq \underline{\alpha} \phi_0, \quad \mathcal{A} \phi_1\geq -\overline{\alpha} \phi_1.
$$
Thus, we have the first inequality with $b=-\overline{\alpha}$. Now, we write again $V=V_0+V_1$ with $V_0:(x,p)\mapsto x^q$ and $V_1:(x,p)\mapsto 1/x^q$ and separate the calculations. On the one hand, we have

$$
\mathcal{A} V_0(x,p) = x^q \left(q\alpha_p -B(x) (1-\theta_0^q - \theta_1^q)\right),
$$
and as $B(x)\to \infty$ as $x$ tends to infinity, there exists $N$ such that for all $x\geq N$,
$$
q\overline{\alpha} -B(x) (1-\theta_0^2 - \theta_1^2) \leq b-1.
$$
Recall that $b=-\overline{\alpha}$ so that $-(b-1)>1$ and $q>1$.
Thus for all $x\geq \max\{1,N\}$, one has 
\begin{align*}
\mathcal{A} V_0(x,p) -(b-1) V(x,p)&\leq -(b-1)x^{-q}\leq -(b-1)x\leq -(b-1)\phi(x,p).
\end{align*}
Therefore $\mathcal{A} V_0 -(b-1) V$ is bounded by $\zeta_0 \phi$ for $\zeta_0$ large enough. On the other hand, we have
\begin{align*}
\mathcal{A} V_1(x,p) 
&\leq V_1(x,p) \left(-q \underline{\alpha} + B(x)(\theta_0^{-q} + \theta_1^{-q} -1) \right) \\
&\leq V_1(x,p) (b -2 +  B(x)(\theta_0^{-q} + \theta_1^{-q} -1) ),
\end{align*}
as $q\geq (\overline{\alpha}+2)/\underline{\alpha}$ and  $b=-\overline{\alpha}$.
Similarly as above, as $\lim_{x\to 0} B(x) = 0$ we obtain that $ \mathcal{A} V_1- (b-1) V$ is bounded by $\zeta_1 \phi$ for some large enough $\zeta_1>0$.
We then obtain the desired result by setting $a=b-1$ and $\zeta= \zeta_0+\zeta_1$.
\end{proof}

Lemma~\ref{lem:psi, V} above almost gives the sufficient drift conditions of \cite[Proposition 2.2]{bansaye2019nonconservative} to verify \cite[Assumption (A0) (A1) (A2)]{bansaye2019nonconservative}. However, a lower bound is missing. The rest of this subsection is dedicated to adapting arguments of the type \cite{CG19} to prove \cite[Assumption (A0) (A1) (A2)]{bansaye2019nonconservative} in our setting.

\iffalse
\begin{rema}
Lemma~\ref{lem:psi, V} above is the only step in the proof of Theorem~\ref{th:mainvp} where we used the assumption
$$
\lim_{x\to 0} B(x) = 0, \qquad \lim_{x\to \infty} B(x) = +\infty.
$$
A careful reading of its proof show that this assumption can be weakened by 
$$
\limsup_{x\to 0} B(x) < c_1, \qquad \liminf_{x\to \infty} B(x) > c_2,
$$
for some constants $c_1<c_2$ that can be computed. By refining the Lyapunuv functions, we could surely further refine these conditions. We did not pursue this way for sake of presentation and because of Assumption~\ref{hyp:B} is realistic for the applications in view.
\end{rema}
\fi

\begin{lemm}
\label{lem:ineg-interm}
Under Assumption \ref{hyp:B}, for every $t\geq 0$, $M_tV$ and $M_t\phi$ are finite. Moreover,
\begin{equation}
\label{eq:phi-mino}
M_t V \leq e^{(a+\zeta)t} V,\qquad M_t \phi \geq e^{bt} \phi,
\end{equation}
and
\begin{equation}
\label{eq:V-majo}
M_t V \leq e^{a t} V + \frac{\zeta}{a-b} M_t \phi,
\end{equation}
where $a,b,\zeta$ are the constants from Lemma~\ref{lem:psi, V}.
\end{lemm}

\begin{proof}
We begin the proof by a standard localization argument (as in \cite{MTIII}) to prove that martingale properties of Lemma~\ref{lem:generateur} extend to non-bounded functions, then we use Gronwall lemma.
Let $m>0$ and set
$$
\tau_m=\inf\left\{ t\geq 0 \ | \  \exists u \in \mathcal{V}_t, \ X^u_t \notin [1/m, m] \right\}.
$$
As $V$ and $\mathcal{A} V$ are bounded over $[1/m, m]$, it follows from Lemma~\ref{lem:generateur} that
$$
\left( Z_{t\wedge \tau_m}(V) - Z_0(V) - \int_0^{t\wedge \tau_m} Z_s(\mathcal{A} V) ds \right)_{t\geq 0}
$$
is a martingale. Now from Lemma~\ref{lem:psi, V}, $\mathcal{A} V \leq C V$ for $C=a+\zeta$, then $(e^{-C (t\wedge \tau_m)} Z_{t\wedge \tau_m}(V))_{t\geq 0}$ is a supermartingale (see \cite[Corollary 3.3 p. 66]{EK09} for instance) %\cmbc{On pourrait calculer un $\mathcal{A}$ espace temps directement au lieu d utiliser ce cor.} and then from Fatou Lemma and Lemma~\ref{lem:existence} 
one obtains
$$
\mathbb{E}_{(x,p)}\left[ e^{-C t} Z_{t}(V)\right] \leq \liminf_{m \to \infty} \mathbb{E}_{(x,p)}\left[ e^{-C (t\wedge \tau_m)} Z_{t\wedge \tau_m}(V)\right] \leq V(x,p).
$$
We deduce from this inequality and from $\phi\leq V$, that $Z_t(V)$ and $Z_t(\phi)$ are integrable and $M_t V \leq e^{(a+\zeta)t} V$. As a consequence, using Lemmas~\ref{lem:generateur}, \ref{lem:psi, V}, the preceding localization argument, and now dominated convergence, %(instead of Fatou Lemma), 
we obtain
\begin{equation}
\label{eq:V-gronw}
M_t V \leq V + \int_0^t (a M_s V + \zeta M_s \phi ) ds,
\end{equation}
and
\begin{equation}
\label{eq:phi-gronw}
M_t \phi \geq \phi + b \int_0^t M_s \phi ds.
\end{equation}
On the first hand, Equation~\eqref{eq:phi-gronw} and Gronwall Lemma entail the second equation in~\eqref{eq:phi-mino}. On the other hand, Equation~\eqref{eq:V-gronw} and Gronwall Lemma entail
$$
M_t V \leq e^{a t} V + \zeta \int_0^t e^{a(t-s)} M_s \phi ds.
$$
Applying operator $M_{s}$ to the second functional inequality in \eqref{eq:phi-mino} taken at $t-s$ yields $M_s \phi\leq e^{-b(t-s)} M_t\phi$ and ends the proof of Equation~\eqref{eq:V-majo}.

\end{proof}

Fix now $\tau>0$ and set 
\begin{align}\label{def:psi}
\psi=  e^{-(a+\zeta)\tau} M_\tau \phi.
\end{align}
We have the straightforward inequality
$$
\psi \leq e^{-(a+\zeta)\tau} M_\tau V \leq  V,
$$
and we can now prove that Equations (\ref{eq:A0A1A2}) hold.
\begin{lemm}
\label{lem:LYAPUNOV}
There exist $\beta>\alpha$, $\theta>0$ and a compact set $K$ such that
\begin{equation}
\label{eq:Vlyap}
M_\tau V \leq \alpha V + \theta \mathbf{1}_K \psi,
\end{equation}
and
\begin{equation}
\label{eq:psilyap}
 M_\tau \psi \geq \beta \psi.
\end{equation}
\end{lemm}
\begin{proof}
Let $K=\{V\leq R \psi\}$, where constant $R$ will be fixed below. Lemma~\ref{lem:ineg-interm} gives \eqref{eq:Vlyap} and \eqref{eq:psilyap} with
$$
\alpha= e^{a \tau} + \frac{1}{R}, \ \beta =e^{b \tau}, \ \theta = \frac{\zeta}{a-b}.
$$
By choosing a sufficiently large $R$, one gets the desired inequality $\beta>\alpha$ (recall that $b>a$ from Lemma~\ref{lem:psi, V}). It remains to prove that $K$ is a compact set. To that end, let us show that
\begin{equation}
\label{eq:limit-K}
\lim_{x \to 0 } \frac{ \psi(x,p)}{V(x,p)} = \lim_{x \to +\infty } \frac{ \psi(x,p)}{V(x,p)} =0.
\end{equation}

From Lemma~\ref{lem:generateur} $(i)$ and Lemma~\ref{lem:ineg-interm}, we have 
\begin{align*}
e^{(a+\zeta)\tau} \psi(x,p)
&=   M_\tau \phi(x,p)\\
&=\phi(xe^{\alpha_p \tau}, p)e^{-\int_{0}^{\tau}B(xe^{\alpha_p s})ds}\\
&+\int_{0}^{\tau}e^{-\int_{0}^{s}B(xe^{\alpha_p s'})ds'} B(xe^{\alpha_p s}) M_{\tau-s}\phi(\theta_0 xe^{\alpha_p s}, 0) ds\\
&+\int_{0}^{\tau}e^{-\int_{0}^{s}B(xe^{\alpha_p s'})ds'} B(xe^{\alpha_p s}) M_{\tau-s}\phi(\theta_1 xe^{\alpha_p s}, 1) ds\\
&\leq\phi(xe^{\alpha_p \tau}, p)e^{-\int_{0}^{\tau}B(xe^{\alpha_p s})ds}\\
&+\int_{0}^{\tau}e^{-\int_{0}^{s}B(xe^{\alpha_p s'})ds'} B(xe^{\alpha_p s})e^{(a+\zeta)(\tau-s)}V(\theta_0 xe^{\alpha_p s}, 0) ds\\
&+\int_{0}^{\tau}e^{-\int_{0}^{s}B(xe^{\alpha_p s'})ds'} B(xe^{\alpha_p s}) e^{(a+\zeta)(\tau-s)}V(\theta_1 xe^{\alpha_p s}, 1) ds.
\end{align*}
Then using Assumption~\ref{hyp:B} and the definitions of $\phi$ and $V$, we obtain the limits~\eqref{eq:limit-K}. 
\end{proof}

%\iffalse
%
%\begin{lemm}
%Fix $\tau>0$, $V$ and $\psi=e^{-(a+\zeta)\tau }M_\tau \phi$. Then $\psi$ verifies \cite[Assumption (A0) (A1) (A2)]{bansaye2019nonconservative}. Moreover $K= \{V\leq R \psi\}$ is included in a compact set of $(0,+\infty)\times\{0,1\}$ for any real number $R$ and $M_t\psi \geq e^{bt }\psi$ for all $t\geq 0$. \cms{R quelconque ou positif ?}
%\end{lemm}
%
%\begin{proof}
%The proof of \cite[Assumption (A0) (A1) (A2)]{bansaye2019nonconservative} is as in the beginning of the step \#2 of the proof of \cite[Theorem 2.1]{CG19}. It remains to show that $K= \{V\leq R \psi\}$ is included in a compact set of $(0,+\infty)\times\{0,1\}$  for any $R$. To this end, it is sufficient to show that\cms{il faudrait au moins rappeler les A0, A1... sinon on comprend vraiment rien}
%$$
%\lim_{x \to 0 } \frac{V(x,p)}{ \psi(x,p)} = \lim_{x \to +\infty } \frac{V(x,p)}{ \psi(x,p)} = +\infty.
%$$
%Fom Duhamel equation \cmt{je ne comprends pas le sens de l'inégalité ici}  Lemma \ref{lem:generateur} (i), we have
%$$
%\psi(x,p) = e^{-(a+\zeta)\tau} M_\tau \phi(x,p) \leq \phi(xe^{\alpha_p \tau}) e^{-\int_0^\tau B(xe^{\alpha_p s} ds)}.
%$$
%\cms{Je ne comprends pas du tout la fin de cette preuve. D'où sort la dernière ligne ? Quel rapport avec V/psi, quel rapprot avec la dernière inégalité de l'énoncé ?}
%\end{proof}
%
%\fi

\begin{rema}
\label{rq:borne-lambda}
When the existence of eigenelements is known, the preceding inequalities allow to give some bounds on eigenvalues $\lambda$. More precisely as $\phi_0 :x\mapsto x$ is bounded by $V$, $Z_t(\phi_0)$ is integrable and from Lemma~\ref{lem:generateur} and
$$
\underline{\alpha} \phi_0 \leq \mathcal{A} \phi_0 \leq \overline{\alpha} \phi,
$$
we find that $(e^{\underline{\alpha} t} Z_t(\phi_0))_{t\geq 0}$ and $(e^{\overline{\alpha} t} Z_t(\phi_0))_{t\geq 0}$ are sub and super-martingales. Integrating over $\mathbb{E}_\gamma$, we find
$$
\underline{\alpha} \leq \lambda \leq \overline{\alpha}.
$$
\end{rema}

%
%%%%%%%%%%%%%%%%%%%%%%%%%%%%%%%%
\subsubsection{Mass ratio inequality}
%%%%%%%%%%%%%%%%%%%%%%%%%%%%%%%%
%
We now show the mass-ratio inequality that states that almost all cells will grow with almost the same speed. This corresponds to \cite[Assumption (A4)]{bansaye2019nonconservative}  or Equation~\eqref{eq:A4}. 

\begin{lemm}\label{lem:massratio}\
For any compact set $\mathcal K\subset (0,+\infty)$, there exist $d>0$ such that for all $(x,m)$ and $(y,p)$ in $K\times \{0,1\}$ and $t\geq 0$, one has
\begin{equation}
\label{eq:A4}
M_{t} \psi(x,m) \geq d M_{t} \psi (y,p).
\end{equation}
%That is \cite[Assumption (A4)]{bansaye2019nonconservative}  holds.
\end{lemm}
\begin{proof}
The proof is based on the approach developed in \cite{CG19}. Let us fix $(x,m)$ and $(y,p)$ in $\mathcal K\times \{0,1\}$ and some time
$$
T> \frac{1}{\underline{\alpha}} \log\left( \frac{\max(\mathcal K)-\min(\mathcal K)}{\min(\theta_0, \theta_1)} \right).
$$ There exists $k \in \mathbb{N}$ and $s\leq T$ such that
$$
y= \theta_0^k \theta_p e^{\alpha_0 (T-s) \alpha_p s} x.
$$
Let use the notation of the path-wise construction introduced in Section~\ref{subse:definition}. Let $v= 01\cdots1$ with $k$ times the digit $0$. We have
$$
M_{T} f(x,m) \geq \mathbb{E}_{(x,m)}\left[\mathbf{1}_{b_v \leq T, d_v\geq T+s } \sum_{u \in \mathcal{V}^v_T} f(X^u_T, P^u_T)\right],
$$
where $\mathcal{V}^v_T$ is the set of individuals that are issued from individual $v$ and alive at time $T$. As $Y^v_{b_v+s} =(y,p)$, the strong Markov property gives
$$
M_T f(x,m) \geq \mathbb{E}_{(x,m)}\left[\mathbf{1}_{b_v \leq T, d_v\geq T+s } M_{T-b_v-s} f(y,p) \right].
$$
Then for $f=M_r \psi$ and $r\geq 0$ we have
$$
M_{r+T} \psi(x,m) \geq M_{r+T} \psi (y,p) \mathbb{E}_{(x,m)}\left[\mathbf{1}_{b_v \leq T, d_v\geq T+s } e^{b(T-b_v-s)} \right] .
$$
Now $(x,y) \mapsto \mathbb{E}_{(x,m)}\left[\mathbf{1}_{b_v \leq T, d_v\geq T+s } e^{b(T-b_v-s)} \right]$ (note that $s$ and $k$ hence $v$ depend on $y$) is a continuous function on a compact set and then has a lower bound $d_0$. Hence the result holds for any $t=r+T\geq T$. %\cmbc{Ici aussi on utilise $B$ continue} 
Now set $t\leq T$. One has
\begin{align*}
M_{t} \psi (y,p) 
&\leq M_t V(y,p) \leq e^{(a+\zeta)t} V(y,p) \leq e^{(a+\zeta)t} \frac{\sup_{\mathcal K} V}{\inf_{\mathcal K} \psi} \psi(x,m) \\
&\leq e^{(a+\zeta)t-bt} \frac{\sup_{\mathcal K} V}{\inf_{\mathcal K} \psi} M_{t} \psi(x,m),
\end{align*}
hence the result also holds true for $t\leq T$.
\end{proof}
%
%%%%%%%%%%%%%%%%%%%%%%%%%%%%%%%%
\subsubsection{Doeblin minoration}
%%%%%%%%%%%%%%%%%%%%%%%%%%%%%%%%
%
In this subsection, we show the Doeblin minoration condition \cite[Assumption (A3)]{bansaye2019nonconservative}. This assumption is an irreducibility and aperiodicity type assumption. 

\begin{lemm}\label{lem : doeblin}
%\cmbc{Attention ici ca ne marche que quand $\alpha_0 \neq \alpha_1$}
For any compact set $\mathcal K \subset (0,+\infty)$, there exist a probability measure $\nu$ and $c>0$ such that
$$
\forall (x,p) \in \mathcal K\times \{0,1\}, \quad \delta_{(x,p)} M_\tau \geq c \nu,
$$
where $\tau$ is defined above Equation (\ref{def:psi}).
\end{lemm}

\begin{proof}

Let $\delta < \min(\mathcal K)$. We show that starting from one cell of size in $\mathcal K$, there is a (possibly very small but) positive probability that at time $\tau$, there is at least one cell which size is uniformly distributed on $I=[\delta' ,\delta]$, for any $\delta'<\delta$ fixed. 

Using the notation of the path-wise construction of Section~\ref{subse:definition}, we have, for any non-negative measurable function $f$,
\begin{align*}
\delta_{(x,p)} M_\tau f
&\geq \sum_{k\geq 0} \mathbb{E}_{(x,p)}\left[f\left( e^{\alpha_p d_\emptyset} e^{\alpha_0 (d_0 -b_0)}e^{\alpha_1(\tau-d_0)} \theta_0\theta_1^k x, 1 \right)\mathbf{1}_{b_{01\cdots1}\leq \tau < d_{01\cdots1}} \right]\\
&= \sum_{k\geq 0} \mathbb{E}_{(x,p)}\left[f\left( e^{b_0(\alpha_p-\alpha_0)}e^{d_0 (\alpha_0 -\alpha_1)}e^{\alpha_1\tau)} \theta_0\theta_1^k x, 1 \right)\mathbf{1}_{b_{01\cdots1}\leq \tau < d_{01\cdots1}} \right],
\end{align*}
where $01\cdots1$ contains one $0$ succeeded by $k$ times $1$. The only source of randomness in the last expectation are $b_0, d_0, b_{01\cdots1}$ and $d_{01\cdots1}$. As $B$ is non-negative, the couple $(b_0,d_0)$ admits a density $(b,d)\mapsto \varphi_x(b,d)$ with respect to the Lebesgue measure. Moreover, This density is positive over the set $$
\{(b,d) \in (0,\tau)^2 \ | \ d>b \},
$$
 and $(b,d,x)\mapsto \varphi_x(b,d)$ is continuous. %\cmbc{Attention je suppose ici que $B$ est continue} 
One also has
 $$
 \mathbb{P}_{(x,p)}(b_{01\cdots1}\leq \tau < d_{01\cdots1})>0,
 $$
 and thus the latter probability is uniformly lower bounded over $\mathcal K$ by a some positive constant (depending on $k$). Then, by a change of variable (on $d_0$), for some constants $c_k,c>0$, one obtains
\begin{align*}
&  \sum_{k\geq 0}  \mathbb{E}_{(x,p)}\left[ f\left( e^{ b_0 (\alpha_p - \alpha_0)} e^{ d_0 (\alpha_0 - \alpha_1) } e^{\alpha_1 \tau} \theta_0\theta_1^k x, 1 \right)\mathbf{1}_{b_{01\cdots1}\leq \tau < d_{01\cdots1}} \right] \\
\geq & \sum_{k\geq 0} c_k \int_{  e^{\alpha_1 \tau} \theta_0\theta_1^k x}^{ e^{ \tau (\alpha_0 - \alpha_1) } e^{\alpha_1 \tau} \theta_0\theta_1^k x} f(u,1) du\\
\geq &c  \int_I \frac{f(u,1)}{\delta-\delta'} du.
\end{align*} 
The result is proved by setting $\nu(f)=\int_I \frac{f(u,1)}{\delta-\delta'} du$.
\end{proof}
%\cmbc{Je ne suis pas fan de comment j'ai rédigé ici}

\begin{rema}
Lemma~\ref{lem : doeblin} is the only step where we used Assumption~\ref{hyp:B} $(ii)$. When $\alpha_0=\alpha_1$, Lemma \ref{lem : doeblin} is not satisfied. We can see in the proof that the change of variable is no longer possible because $d \mapsto e^{d \times 0}$ is constant. As shown in \cite{gabriel2019periodic,bernard2016cyclic}, in such a model, the distribution of the process is concentrated in a comb that depends on the initial conditions and then does not verify the Doeblin assumption. This is why eigenelements exist (see \cite{doumic2010eigenelements}) but the convergence does not hold (see \cite{gabriel2018steady}). %From That can be interpreted as a justification for asymmtery.
\end{rema}

Wa can now turn to the proof of Theorem~\ref{th:mainvp}.\\

\noindent \textit{Proof of Theorem~\ref{th:mainvp}}.
Theorem~\ref{th:mainvp} is now a consequence of \cite[Theorem 2.1]{bansaye2019nonconservative} and Lemmas~\ref{lem:LYAPUNOV}, \ref{lem:massratio} and \ref{lem : doeblin} that establish that \cite[Assumptions \textbf{A}]{bansaye2019nonconservative} hold in our context.
\hfill $\Box$
%
%%%%%%%%%%%%%%%%%%%%%%%%%%%%%
\section{Variations of the principal eigenvalue}
\label{se:darwin}
%%%%%%%%%%%%%%%%%%%%%%%%%%%%%
%
This section is dedicated to the study of the variations of the principal eigenvalue $\lambda$ from Theorem \ref{th:mainvp} with explicit formulas in the special case $B(x)=x$.
%\cmbc{On pourrait mettre les résultats sur $B(x)=x$ dans une section dédié ou pas}\cmt{Le résultat principal dans le cas B(x)=x concerne la monotonie de la valeur propre, donc il me semble que ça a sa place dans cette section}. 
In particular, it contains the proof of Theorems \ref{pr:variationmalthus} and \ref{th:inegpartial}. More specifically, recall the decompositions $\alpha_1=\alpha + \epsilon$, $\alpha_0= \alpha -\epsilon$ in a such way that $\alpha = (\alpha_0 +\alpha_1)/2$ and $\epsilon = (\alpha_1 -\alpha_0)/2$, $\theta=\theta_0$. Theorems \ref{pr:variationmalthus} and \ref{th:inegpartial} describe the variations of the map $\uu = (\alpha,\epsilon, \theta)\in(0, +\infty)^2\times (0,1)\mapsto \lambda$ and show that under some suitable assumptions, asymmetry is optimal.
%
%\iffalse
%\begin{align*}
%&\mathcal{A}_u h_u = \lambda_u h_u, \quad \mathcal{A}^{*}_u \gamma_u =\lambda_u \gamma_u, \quad \int h_u \gamma_u =1,\quad \int \gamma_u = 1. 
%\end{align*}
%We are interested in the variation of the principal eigenvalue $u\mapsto \lambda_u$. For all $u=(\alpha,\epsilon,\theta)\in (0,+\infty)^2\times (0,1)$ we note $\lambda(u)=\lambda_u$.\cms{plein de notations pas définies ici, à reprendre}
%
%
% Let us recall that for any $\uu\in (\alpha,\epsilon, \theta)\in(0, +\infty)^2\times (0,1)$ the eigenelements $(\lambda_{\uu},h_{\uu},\gamma_{\uu})$ satisfies
%\begin{equation*}
%\mathcal{A}_{\uu}h_{\uu} = \lambda_{\uu} h_{\uu},\quad \gamma_{\uu}\mathcal{A}_{\uu}=\lambda_{\uu}\gamma_{\uu},\quad \int\gamma_{\uu}=1,\quad\gamma_{\uu}(h_{\uu})=1
%\end{equation*}
%where
%
%\begin{equation}\label{eq:generateurvariation}
%\mathcal{A}_{\uu}f(x,p)=(\alpha - (1-2m)\epsilon)x\partial_x f(x,p) + B(x)(f(\theta x,0) + f((1-\theta)x, 1) - f(x,p)).
%\end{equation}
%\fi

This section is organized as follows. We begin by proving the regularity of eigenelements and then proove Theorem \ref{pr:variationmalthus} in Section \ref{subse:var}. Then we give a general formula for the eigenmeasure in Section \ref{sec:eigenmeasure}. Finally, we conclude this part by studying the special case where $B$ is the function $x\mapsto x$ and by proving Theorem \ref{th:inegpartial} in Section \ref{sec:Bidentity}.
%
%%%%%%%%%%%%%%%%%%%%%%%%%%%%%%%%
\subsection{Regularity of the eigenelements and proof of Theorem \ref{pr:variationmalthus}}
%\label{subse:regh}
%\subsection{Variation of $\lambda$ : proof of Theorem \ref{pr:variationmalthus}}
\label{subse:var}
%%%%%%%%%%%%%%%%%%%%%%%%%%%%%%%%
%
%\cmbc{En fait, j'ai viré partout qu'on utilisait la méthode de \cite{michel2006optimal} car "sa méthode" c'est juste de dire que pour deriver $uv$ on ecrit $d(uv)=vdu+udv$ c'est un peu simple et la première chose à faire non?  }
%\cmbc{On avait noté partout les intégrale sous la forme $\mu(f)$ j'ai gardé ca pour être consistant}

We begin by proving that the eigenfunction $h$ defined in Theorem \ref{th:mainvp} is smooth enough.

\begin{lemm}
\label{lem:reghx}
Under Assumption~\ref{hyp:B}, $h$ is in $C^1$ and
$$
\mathcal{A} h= \lambda h.
$$
Moreover, if $B$ is in $C^p$ then $h$ is in $C^{p+1}$.
%\iffalse
%$h\in C^{\infty}$ et pour tout compact $K$ 
%$$
%\sup_{x \in K, p} |h(x,p)| + |\partial_x h(x,p)| + |\partial_x^2 h| \leq F(paramètre) 
%$$
%\fi
\end{lemm}

\begin{proof}
We begin by showing that $x\mapsto h(x,p)$ is continuous. From Lemma~\ref{lem:generateur} $(i)$, we have
\begin{align*}
e^{\lambda t} h(x,p) &= h(x e^{\alpha_p t},p) e^{\int_0^t B(xe^{\alpha_p} s ) ds }\\
&+\int_0^t B(xe^{\alpha_p s} )  e^{\int_0^s B(xe^{\alpha_p} s' ) ds' } e^{\lambda (t-s)} (h(\theta_0 x  e^{\alpha_p s}, 0) + h(\theta_0 x  e^{\alpha_p s}, 1))ds.
\end{align*}
Thus the function $\Xi$ defined by
\begin{align*}
\Xi(t,x)
&:=\int_0^t B(xe^{\alpha_p s} )  e^{\int_0^s B(xe^{\alpha_p s'} ) ds' } e^{\lambda (t-s)} (h(\theta_0 x  e^{\alpha_p s}, 0) + h(\theta_1 x  e^{\alpha_p s}, 1))ds
\end{align*}
verifies
$$
\Xi(t,x) \leq \int_0^t B(xe^{\alpha_p s} )  e^{\int_0^s B(xe^{\alpha_p s'} ) ds' } e^{\lambda (t-s)} ( V(\theta_0 x  e^{\alpha_p s})+ V(\theta_1 x  e^{\alpha_p s}))ds,
$$
as $h\leq V$.
This function then tends to $0$ when $t\to 0$. Then fixing $y>0$ and choosing $t$ such that $y=x e^{\alpha_p t}$, we get 
$$
h(x,p) = h(y,p) a(x/y,x) + b(x/y,x)
$$
where $a,b$ are two functions satisfying %\cmbc{Je vais vite ou ca va?}
$$
\lim_{y\to x}  a(x/y,x) =1, \qquad \lim_{y \to x} b(x/y,x)=0.
$$
Thus $x\mapsto h(x,p)$ is continuous as $h(y,p)$ tends to $h(x,p)$ as $y$ trends to $x$. The proof of the differentiation is similar. Indeed, we have
$$
h(xe^{\alpha_p t}) = e^{\lambda t} h(x,p) -\int_0^t B(xe^{\alpha_p} s )  e^{\int_0^s B(xe^{\alpha_p} s' ) ds' } e^{\lambda (t-s)} ( h(\theta_0 x  e^{\alpha_p s})+ h(\theta_1 x  e^{\alpha_p s}))ds.
$$
As $h$ is continuous, we can differentiate the right member of the right hand side. Then we obtain that $h$ can be differentiated and one has (punctually) $\mathcal{A} h= \lambda h$. This yiels that $\partial_x h$ is continuous and then $h$ is $C^1$.
The functional equation $\mathcal{A} h= \lambda h$ and the inequality $h\leq V$ permit to bound $\partial_x h$ and then dominated convergence implies the last regularity property.
\end{proof}

\iffalse
For any measurable function $f$ and any measure $\mu$, we denote 
\begin{equation*}
\langle f,\mu\rangle =\int_{\lbrace 0,1\rbrace} \int_{0}^{+\infty}f(x,p)\mu(dx,dp).
\end{equation*}
Let  $\uu,\mathbf{v}\in (0, +\infty)^2\times (0,1)$.

 We have
\begin{align*}
\lambda_{\uu}-\lambda_{\mathbf{v}} &= \gamma_\mathbf{u}\left(\mathcal{A}_\mathbf{u}h_\mathbf{u} \right) - \gamma_\mathbf{v} \left( \mathcal{A}_\mathbf{v} h_\mathbf{v} \right) \\
&=\langle \mathcal{A}_\mathbf{u}h_\mathbf{u}, \gamma_\mathbf{u} - \gamma_{\vv}\rangle +\langle (\mathcal{A}_{\uu}-\mathcal{A}_{\vv})h_{\uu},\gamma_{\vv}\rangle +\langle \mathcal{A}_{\vv}(h_{\uu}-h_{\vv}),\gamma_{\vv}\rangle\\
&=\lambda_{\uu} \langle h_{\uu},\gamma_{\uu}-\gamma_{\vv}\rangle + \langle (\mathcal{A}_{\uu}-\mathcal{A}_{\vv})h_{\uu},\gamma_{\vv}\rangle +\lambda_{\vv} \langle h_{\uu} - h_{\vv}, \gamma_{\vv}\rangle.
\end{align*}
From the last equality we deduce that
\begin{equation}\label{eq:variationvaleurpropre}
\lambda_{\uu} - \lambda_{\vv} =\frac{\langle (\mathcal{A}_{\uu} - \mathcal{A}_{\vv})h_{\uu},\gamma_{\vv}\rangle}{\langle h_{\uu},\gamma_{\vv}\rangle}.
\end{equation}\fi

 From now on, we study the eigenelements $(\lambda,\gamma,h)$ from Theorem \ref{th:mainvp} as functions of $\mathbf{u}$. We highlight this dependence on the parameter $\mathbf{u}$ by denoting $\mathcal{A}_{\uu}$ instead of $\mathcal{A}$ the \textit{extended generator} of the semigroup $(M_t)_{t\geq 0}$ defined in \eqref{eq:generateur} (the term extended generator is used from Lemma~\ref{lem:generateur} and definitions in \cite[Section 1.3]{MTIII},\cite{Davis93} which are closely related). It is defined for $C^1$ functions $f$ by
\begin{equation}\label{eq:generateurvariation}
\mathcal{A}_{\uu}f(x,p)=(\alpha - (1-2p)\epsilon)x\partial_x f(x,p) + B(x)(f(\theta x,0) + f((1-\theta)x, 1) - f(x,p)).
\end{equation}
Similarly, we denote by $(\lambda_{\uu},h_{\uu},\gamma_{\uu})$ the eigenelements of Theorem~\ref{th:mainvp} and also use the notation $\lambda(\uu)=\lambda_{\uu}$.
Using Lemma~\ref{lem:reghx}, we can now prove the continuity of the eigenvectors with respect to the parameters $\uu$.

\begin{lemm}\label{lem:continuityeigen}
If $B$ is in $C^2$ then the maps  $\uu\mapsto \lambda_{\uu}$, $\uu\mapsto h_{\uu}$ and  $\uu \mapsto \gamma_{\vv}(h_{\uu})$ are continuous on $(0,+\infty)^2\times (0,1)$ for every $\vv\in(0,+\infty)^2\times (0,1)$.
\end{lemm}

\begin{proof}
Let $(\uu_n)$ be any sequence converging to some fixed $\uu \in (0,+\infty)^2\times (0,1)$ . We will show that $(\lambda_{\uu_n}, h_{\uu_n})$ tends to $(\lambda_{\uu}, h_{\uu})$ as $n$ tends to infinity.

By Remark~\ref{rq:borne-lambda}, $(\lambda_{\uu})$ is bounded and by Lemma~\ref{lem:reghx}, $h_{\uu_n}$ is in $C^2$. Using $\mathcal{A}_{\uu_n} h_{\uu_n} = \lambda_{\uu_n} h_{\uu_n}$ and $h_{\uu_n}\leq V$, we can bound $h_{\uu_n}(x,p), \partial_x h_{\uu_n}(x,p)$ and $\partial^2_x h_{\uu_n}(x,p)$ locally in $x$ uniformly in $n$. Then Arzelà–Ascoli theorem ensures that $(\lambda_{\uu_n}, h_{\uu_n}, \partial_x h_{\uu_n})$ is relatively compact (for the compact convergence). Taking the (punctual) limit in $\mathcal{A}_{\uu_n} h_{\uu_n} = \lambda_{\uu_n} h_{\uu_n}$ shows that each adherence point $(\lambda, h, \partial_x h)$ of this sequence verifies $\mathcal{A}_{\uu} h = \lambda h$ and $h\leq V$. Then uniqueness of eigenelements shows that $\lambda=\lambda_{\uu}$ and $h=h_{\uu}$.
Finally the last statement comes from dominated convergence.
\end{proof}

We can now differentiate the eigenvalue; namely we are now able to prove Theorem \ref{pr:variationmalthus}.

\begin{proof}[Proof of Theorem \ref{pr:variationmalthus}]
The proof is based on the equality
\begin{equation}\label{eq:variationvaleurpropre}
\lambda_{\uu} - \lambda_{\vv} =\frac{\gamma_{\vv} \left( (\mathcal{A}_{\uu} - \mathcal{A}_{\vv})h_{\uu} \right)}{\gamma_{\vv}(h_{\uu})}.
\end{equation}
%\begin{itemize}
(i) We note $\uu=(\alpha+\delta,\epsilon,\theta)$ and $\vv=(\alpha,\epsilon,\theta)$. We have
\begin{align*}
\gamma_{\vv}( (\mathcal{A}_{\uu} - \mathcal{A}_{\vv})h_{\uu}) = \int_{0}^{+\infty}\int_{\lbrace 0,1\rbrace}\delta  x\partial_x h_{\uu}(x,p)\gamma_{\vv}(dx,dp),
\end{align*}
that allows us to conclude by dividing by $\delta$, taking the limit $\delta\rightarrow 0$ and using Lemma \ref{lem:continuityeigen}.\\
(ii) We note $\uu =(\alpha,\epsilon + \delta,\theta)$ and $\vv = (\alpha,\epsilon,\theta)$. We have
\begin{align*}
\gamma_{\vv} ((\mathcal{A}_{\uu} - \mathcal{A}_{\vv})h_{\uu}) &=\int_{0}^{+\infty}\left[-(\epsilon+\delta)x\partial_x h_{\uu}(x,0)+\epsilon x\partial_x h_{\uu}(x,0)\right]\gamma_{\vv}(dx,0)\\
&+\int_{0}^{+\infty}\left[(\epsilon+\delta)x\partial_x h_{\uu}(x,1)-\epsilon x\partial_x h_{\uu}(x,1)\right]\gamma_{\vv}(dx,1)\\
&=\delta \left(\int_{0}^{+\infty}x\partial_xh_{\uu}(x,1)\gamma_{\vv}(dx,1)-\int_{0}^{+\infty}x\partial_x h_{\uu}(x,0)\gamma_{\vv}(dx,0)\right),
\end{align*}
that allows us to conclude by dividing by $\delta$, taking the limit $\delta\rightarrow 0$ and using Lemma \ref{lem:continuityeigen} again.\\
(iii) We note $\uu =(\alpha,\epsilon,\theta+\delta)$ and $\vv = (\alpha,\epsilon,\theta)$ We have
\begin{align*}
(\mathcal{A}_{\uu} - \mathcal{A}_{\vv})h_{\uu}(x,m)&=B(x)\left( h_{\uu}((\theta +\delta)x,0)+h_{\uu}((1-\theta-\delta)x,1) - h_{\uu}(x,m)\right)\\
&-B(x)\left( h_{\uu}(\theta x,0)+h_{\uu}((1-\theta)x,1) - h_{\uu}(x,m)\right).
\end{align*}
We conclude by integrating over $\gamma_{\vv}$, by taking the limit $\delta\rightarrow 0$ and using Lemma \ref{lem:continuityeigen} one last time.
%\end{itemize}
\end{proof}
%
%%%%%%%%%%%%%%%%%%%%%%%%%%%%%%%%
\subsection{Explicit eigenmeasure: a general formula}
\label{sec:eigenmeasure}
%%%%%%%%%%%%%%%%%%%%%%%%%%%%%%%%%%%
%
Before focusing on the special case $B(x)=x$, let us establish here a link between the limiting distribution of the classical symmetric model (as in \cite{michel2006existence, hall1989functional} for instance) and our asymmetric model.

\begin{lemm}\label{lem:Utogamma}
When $\epsilon=0$ (that is $\alpha_0=\alpha_1$), the eigenmeasure $\gamma_{\uu}$ is given by 
$$
\gamma_{\uu}(dx,dp)=\gamma_{\uu}^0(x) \delta_0(dp) + \gamma_{\uu}^1(x) \delta_1(dp),
$$% \cmbc{verifier les sens x,p}
with 
\begin{align*}
\gamma_{\uu}^0(x) & =e^{-\int_{1}^{x}\frac{B(r)+2\alpha}{\alpha r}dr}\int_{0}^{x/\theta}e^{\int_1^{y\theta}\frac{B(r)+2\alpha}{\alpha r}dr}\frac{1}{\alpha \theta y}B(y)\mathcal{U}_{\uu}(y)dy,\\
\gamma_{\uu}^1(x) & =e^{-\int_{1}^{x}\frac{B(r)+2\alpha}{\alpha r}dr}\int_{0}^{x/(1-\theta)}e^{\int_1^{y(1-\theta)}\frac{B(r)+2\alpha}{\alpha r}dr}\frac{1}{\alpha (1-\theta) y}B(y)\mathcal{U}_{\uu}(y)dy,
\end{align*}
where $\mathcal{U}_{\uu}$ is the density of the eigenmeasure for the one-population symmetric model, $i.e.$ it satisfies:
\begin{align}\label{eq:eigenmeasuremono}
\alpha x \mathcal{U}_{\uu}'(x) +(2\alpha +B(x))\mathcal{U}_{\uu}(x)&=\frac{1}{1-\theta}B\left(\frac{x}{1-\theta}\right)\mathcal{U}_{\uu}\left(\frac{x}{1-\theta}\right) \nonumber\\
& + \frac{1}{\theta}B\left(\frac{x}{\theta}\right)\mathcal{U}_{\uu}\left(\frac{x}{\theta}\right),
\end{align}
and $\int_0^{+\infty} \mathcal{U}_{\uu}(x)dx=1$.
\end{lemm}

Lemma~\ref{lem:Utogamma} is a cornerstone in the proof of Theorem~\ref{th:inegpartial}. In addition it is interesting by itself. Indeed, it can trivially be generalized for random divisions (random $\theta$) and using results of \cite{hall1989functional, hall1990functional}, we can exhibit some explicit formulas for $\gamma$ for explicit distributions of $\theta$ (for instance $\theta$ uniformly distributed in $(0,1)$).

\begin{proof}
In this proof we denote the eigenmeasure (by abuse of notation):
\begin{equation*}
\gamma_{\textbf{u}}(dx,dp)=\gamma_{\textbf{u}}^0(x)dx\delta_{0}(dp)+\gamma_{\textbf{u}}^1(x)dx\delta_{1}(dp).
\end{equation*}
Since $\uu=(\alpha, 0,\theta)$ we have $\lambda_{\uu}=\alpha$. Using Equation (\ref{eq:generateurvariation}) with $f(x,0)=0$ and $f(x,1)=f(x)$ we obtain:
\begin{align*}
\alpha \int_{0}^{+\infty}f(x)\gamma_{\uu}^1(x)dx &=\int_{0}^{+\infty}\mathcal{A}_{\uu} f(x,1)\gamma_{\uu}^1(x)dx + \int_{0}^{+\infty}\mathcal{A}_{\uu} f(x,0)\gamma_{\uu}^0(x)dx\\
&=\int_{0}^{+\infty}(\alpha x\partial_x f(x)+B(x)(f((1-\theta)x)- f(x)))\gamma_{\uu}^1(x)\\
&+\int_{0}^{+\infty} B(x)f((1-\theta)x)\gamma_{\uu}^0(x)dx\\
&=\int_{0}^{+\infty}f(x)\left(-\alpha(\gamma_{\uu}^1(x) +x\partial_x\gamma_{\uu}^{1}(x))  - B(x)\gamma_{\uu}^1(x)\right.\\
&\left. +\frac{1}{1-\theta}B\left(\frac{x}{1-\theta}\right)\gamma_{\uu}^1\left(\frac{x}{1-\theta}\right) \right.\\
& \left. + \frac{1}{1-\theta}B\left(\frac{x}{1-\theta}\right)\gamma_{\uu}^0\left(\frac{x}{1-\theta}\right)\right)dx.
\end{align*}
We deduce that
\begin{align*}
\alpha x \partial_x\gamma_{\uu}^1(x)+(2\alpha + B(x)) \gamma_{\uu}^1(x)&=\frac{1}{1-\theta}B\left(\frac{x}{1-\theta}\right)\gamma_{\uu}^1\left(\frac{x}{1-\theta}\right)\\ 
&+ \frac{1}{1-\theta}B\left(\frac{x}{1-\theta}\right)\gamma_{\uu}^0\left(\frac{x}{1-\theta}\right).
\end{align*}
We write $\mathcal{U}_{\uu}(x)=\gamma_{\uu}^1(x)+\gamma_{\uu}^0(x)$ and we conclude by solving the previous equation. In particular the measures $(\gamma_{\uu}^0,\gamma_{\uu}^1)$ are absolutely continuous with respect to the Lebesgue measure.
\end{proof}

\begin{lemm}\label{lem:hunique}
When $\epsilon=0$ (that is $\alpha_0=\alpha_1$), we have, for every $x>0$,
$$
h(x,0)=h(x,1)=x.
$$
\end{lemm}
\begin{proof}
By linearity $h=h(\cdot,0)+h(\cdot,1)$ is a positive eigenvector of the classical symmetric equation and then, by uniqueness, $h:x\mapsto x$; see \cite{doumic2010eigenelements} for details. Now, for $i\in \{0,1\}$ and $x>0$, we set
$$
g(x,i)=-g(x,1-i) =h(x,i)-h(x).
$$
We will show $g_i=0$. On the one hand, as $h(\cdot, i) \geq 0$, $g(\cdot,i)=-g(\cdot,1-i)$, and $h(0)=0$, we necessarily have $g(0,i)=0$.
Using the eigenvector equations, we have
\begin{align*}
\alpha g_0(x)=\alpha x  g'_0(x) + B(x) (g_0(\theta_0 x, 0) -g_0(\theta_1 x,1)- g_0(x)),
\end{align*}
and by integration,
\begin{align*}
g_0(x)= \int_0^x \frac{(\alpha+B(u)) g_0(u)+ B(u) g_0(\theta_1 u)-B(u) g_0(\theta_0 u) }{\alpha u}. du
\end{align*}
Then, there exists $C>0$, such that for any $\delta \in [0,1]$,
$$
\zeta(\delta):=\sup_{x\in [0,\delta]} \frac{|g_0(x)|}{\sqrt{x}} \leq C \sqrt{\delta} \zeta(\delta),
$$
where $C$ does not depend on $\delta$ nor $g_0$ but only on $\sup_{[0, 1/\min(\theta_0,\theta_1)]} B$ and $\alpha$. Consequently if $C \sqrt{\delta} <1$ then $\zeta(\delta)=0$. Thus, $g(x)=0$ on $[0,\delta]$. Iterating this argument, we find $g_0(x)=0$ for all $x>0$.
%Si $p>q$
\end{proof}
%
%%%%%%%%%%%%%%%%%%%%%%%%%%%%%%%%
\subsection{The particular case $B(x)=x$}
\label{sec:Bidentity}
%%%%%%%%%%%%%%%%%%%%%%%%%%%%%%%%
%
In all this section, we assume that $B$ is the function $x\mapsto x$, which verifies all our assumptions. To prove our main result Theorem~\ref{th:inegpartial}, we give some explicit formulas in this special case. More precisely Theorem~\ref{prop:expliciteigen} gives an explicit expression of $\gamma$,  Lemma~\ref{lem:moments} its moments and Lemma~\ref{le:logmoment} its \textit{logarithmic} moments. The section finishes by the proof of Theorem~\ref{th:inegpartial}. All these explicit results are not necessary for proving Theorem~\ref{th:inegpartial} but have an interest by themselves, to derive statistical estimators by the method of moments for example. A reader who is only interested in the proof of the Theorem~\ref{th:inegpartial} can therefore directly jump to the end of this section.\\

%\subsubsection{Eigenmeasure of the one-population model}
We are interested in the local behavior of $\lambda$ at $\uu=(\alpha,0,\theta)$. By Lemma~\ref{lem:Utogamma}, to explicit $\gamma$ at this point, it is enough to study the eigen-measure $\mathcal{U}_{\uu}$ of the physiological symmetric model. 
%\cms{On est dans le cas symétrique dans cette partie ????}
Namely, the solution $\mathcal{U}_{\uu}$ to the integro-differential equation (\ref{eq:eigenmeasuremono}).
%\begin{align}\label{eq:eigenmeasuremono}
%\alpha x \mathcal{U}_{\uu}'(x) +(2\alpha +B(x))\mathcal{U}_{\uu}(x)&=\frac{1}{1-\theta}B\left(\frac{x}{1-\theta}\right)\mathcal{U}_{\uu}\left(\frac{x}{1-\theta}\right)  \\
%&\qquad + \frac{1}{\theta}B\left(\frac{x}{\theta}\right)\mathcal{U}_{\uu}\left(\frac{x}{\theta}\right).\nonumber
%\end{align}
%
\begin{theo}%\cmbc{La formule est bonne Tristan?}
%\cmt{Oui}
\label{prop:expliciteigen}
The solution $\mathcal{U}_{\uu}$ of Equation \eqref{eq:eigenmeasuremono} is given by:
\begin{equation}
\label{eq:UBx}
\mathcal{U}_{\uu}(x)=\frac{K}{\alpha x^2}\sum_{n=0}^{+\infty}\mathbb{E}\left[\prod_{k=1}^{n}\left(\frac{1}{1-\frac{1}{Q_k}}\frac{1}{Q_n} \exp\left(-\frac{1}{Q_n} \frac{x}{\alpha}\right)\right)\right],
\end{equation}
where $K$ is a normalizing constant, $Q_k =\prod_{j=1}^{k}\Theta_j$ where $(\Theta_j)_{j\geq 1}$ is an \textit{i.i.d.} sequence of random variables with distribution $\mathbb{P}(\Theta =\theta)=\theta$ and $\mathbb{P}(\Theta =1-\theta)=1-\theta$. 
\end{theo}

This result generalizes part of the results in \cite{hall1989functional,hall1990functional}. As we will see in its proof, using \cite{guillemin2004aimd}, we can also simplify this expression in some special cases. However, we will not use this explicit expression to calculate moments of $\gamma$.

\begin{proof}
We have
\begin{equation*}
(\alpha x\mathcal{U}_{\uu}(x))' + (x +\alpha)\mathcal{U}_{\uu}(x)=\frac{1}{\theta}B\left(\frac{x}{\theta}\right)\mathcal{U}_{\uu}\left(\frac{x}{\theta}\right)+\frac{1}{1-\theta}B\left(\frac{x}{1-\theta}\right)\mathcal{U}_{\uu}\left(\frac{x}{1-\theta}\right).
\end{equation*}
By multiplying by $x$ we obtain
\begin{equation*}
\alpha x^2\mathcal{U}_{\uu}'(x)+2\alpha x \mathcal{U}_{\uu}(x) + x^2 \mathcal{U}_{\uu}(x)=\frac{x}{\theta}B\left(\frac{x}{\theta}\right)\mathcal{U}_{\uu}\left(\frac{x}{\theta}\right)+\frac{x}{1-\theta}B\left(\frac{x}{1-\theta}\right)\mathcal{U}_{\uu}\left(\frac{x}{1-\theta}\right).
\end{equation*}
We define $Z(x)= A x^2\mathcal{U}_{\uu}(x)$ where $A$ is a normalising constant; we obtain:
\begin{equation*}
\alpha Z'(x)+Z(x)=Z\left(\frac{x}{\theta}\right) +Z\left(\frac{x}{1-\theta}\right).
\end{equation*}
We now take the Laplace transform $\overline{Z}(z)=\int_{0}^{+\infty}e^{-zx}Z(x) dx$ to obtain
\begin{equation}\label{eq:laplace}
\overline{Z}(z)=\frac{1}{\alpha z +1}(\theta \overline{Z}(\theta z)+ (1-\theta)\overline{Z}((1-\theta)z)).
\end{equation}
Let $I$ be a random variable with distribution $Z$. Equation (\ref{eq:laplace}) is equivalent to the following equality in distribution:
\begin{equation*}
I\overset{d}{=}E +\Theta I,
\end{equation*} 
where $\mathbb{P}(\Theta=\theta)=\theta$, $\mathbb{P}(\Theta=1-\theta)=1-\theta$ and $E$ is exponentially distributed with parameter $1/\alpha$. This equation was studied in particular in \cite[Section 3]{guillemin2004aimd}. By using \cite[Proposition 5]{guillemin2004aimd} with the following notation, coming from their setting: $\beta\in (0,1)$, $X=\ln(\Theta)/\ln(\beta)$ and let $E_0\sim \mathcal{E}(1)$, we obtain
\begin{equation*}
\frac{I}{\alpha}\overset{d}{=}E_0+\beta^{X} \frac{I}{\alpha}.
\end{equation*}
which yields Equation~\eqref{eq:UBx}.
\end{proof}

Let us continue now with two lemmas on the calculation of moments of the eigenmeasure.
\begin{lemm}
\label{lem:moments}
Set $m_p=\int_{0}^{+\infty}x^{p}\mathcal{U}_u(x)dx$, for $p\in \mathbb{N}$. We have
\begin{equation*}
m_0=1,\quad m_1=\alpha,\quad m_2 = \frac{-\alpha^2}{\theta \log(\theta)+(1-\theta)\log(1-\theta)},
\end{equation*}
and for all $p>2$:
\begin{equation*}
m_p=\frac{-\alpha^2}{\theta \log(\theta)+(1-\theta)\log(1-\theta)}\prod_{q=2}^{p-1}\left(\frac{\alpha(q-1)}{1-\theta^{q}-(1-\theta)^{q}}\right).
\end{equation*}
\end{lemm}

\begin{proof}
Since $\mathcal{U}_{\uu}$ is the density of a probability measure, we have $m_0=1$. Now, let us define the generator
\begin{equation}\label{eq:generateuronepop}
\mathcal{B}f(x)=\alpha x f'(x) + B(x)(f(x\theta)+f((1-\theta)x)-f(x).
\end{equation}
Then, for every $C^1$ function we have %\cmbc{C'est peut-être un peu rapide mais le lecteur s'est habitué, non?}
\begin{equation*}
\int_0^\infty \mathcal{B} f(x) \mathcal{U}_{\uu}(x) dx =\alpha \int_0^\infty \mathcal{U}_{\uu}(x) f(x) dx.%,\quad \int_{0}^{+\infty}\mathcal{U}_{\uu}(x)dx=1.
\end{equation*}
Using now, $f:x\mapsto 1$ gives
\begin{align*}
\alpha=\alpha\int_{0}^{+\infty}\mathcal{U}_{\uu}(x)dx=\int_{0}^{+\infty}B(x)\mathcal{U}_{\uu}(x)dx=\int_{0}^{+\infty}x\mathcal{U}_{\uu}(x)dx 
\end{align*}
and so $m_1=\alpha$. Now for $p\geq 1$ and $f:x\mapsto x^p$, we have
\begin{align*}
\alpha m_p
&=\alpha\int_{0}^{+\infty}x^{p}\mathcal{U}_{\uu}(x)dx \\
& = \int_{0}^{+\infty}\mathcal{U}_{\uu}(x)\left(\alpha x px^{p-1}+x(\theta^p x^p +(1-\theta)^p x^p - x^p)      \right)dx\\
&= \alpha (p-1) m_p + m_{p+1} (\theta^p + (1-\theta)^p-1),
% & = \int_{0}^{+\infty}\mathcal{U}_{\uu}(x)\left(\alpha (p-1)x^p +x^{p+1}(\theta^p + (1-\theta)^p - 1)   \right)dx.
\end{align*}
which yields for $p>1 1$,
\begin{equation}
\label{eq:mrec}
m_{p+1}=\frac{\alpha(p-1)}{1-\theta^p -(1-\theta)^p}m_p.
\end{equation}
To complete the demonstration, we need to calculate $m_2$. However, Equation~(\ref{eq:mrec}) holds for real numbers $p$, and then taking the limit $p\to 1$ ends the proof.
\end{proof}

%Even if we do not need the next Lemma to prove Theorem~\ref{th:inegpartial}, we expose it to show that the strategy  of Lemma~\ref{lem:moments} is useful for others functional than moments.

\begin{lemm}\label{le:logmoment}
For all $p\in \mathbb{N}$, let us define $l_p=\int_{0}^{+\infty}x^p\log(x)\mathcal{U}_{\uu}(x)dx$. Then we have
\begin{equation*}
l_0=1+ 2 (\log(\theta)+\log(1-\theta)),\quad l_1 = \log(\theta)\alpha +\log(1-\theta)\alpha,
\end{equation*}
and for all $p\in \mathbb{N}^{*}$,
\begin{equation*}
l_{p+1}=\frac{1}{\theta^p +(1-\theta)^p}\left(\alpha(1-p)l_p-m_{p+1}(\theta^p \log(\theta)+(1-\theta)^p\log(1-\theta)\right) - \alpha m_p).
\end{equation*}
\end{lemm}
\begin{proof}
The proof is simiular to that of Lemma~\ref{lem:moments} using $f:x \mapsto x^p \log(x)$.
\iffalse
As in Lemma~\ref{lem:moments}, using $f:x\mapsto \log(x)$, we show
\begin{equation*}
\alpha\int_{0}^{+\infty}\log(x)\mathcal{U}_{\uu}(x)dx=\alpha + \int_{0}^{+\infty}B(x)(\log(x)+\log(\theta)+\log(1-\theta))\mathcal{U}_{\uu}(x)dx
\end{equation*}
and we deduce that
\begin{equation}\label{eq:log}
\alpha\int_{0}^{+\infty}\log(x)\mathcal{U}_{\uu}(x)dx=\alpha + \int_{0}^{+\infty}x\log(x)\mathcal{U}_{\uu}(x)dx +\alpha (\log(\theta)+\log(1-\theta)).
\end{equation}
Now let $g(x,p)=\log(x)\mathbf{1}_{p=0}$. We use that $\gamma_{\uu}\mathcal{A}_{\uu}g = \alpha \gamma_{\uu} g$ and we obtain
\begin{equation*}
\log(\theta)\int_{0}^{+\infty}x\gamma^0_{\uu}(x)dx = -\int_{0}^{+\infty}x\log(\theta x)\gamma_{\uu}^1(x)dx,
\end{equation*}
and so
\begin{equation*}
\log(\theta)\int_{0}^{+\infty}x\mathcal{U}_{\uu}(x)dx=\int_{0}^{+\infty}x\log(x)\gamma_{\uu}^1(x)dx,
\end{equation*}
\begin{equation*}
\log(1-\theta)\int_{0}^{+\infty}x\mathcal{U}_{\uu}(x)dx=\int_{0}^{+\infty}x\log(x)\gamma_{\uu}^0(x)dx.
\end{equation*}
Hence, by suming the two previous equations, we obtain $l_1 = \log(\theta)\alpha +\log(1-\theta)\alpha$ and Equation (\ref{eq:log}) gives that $l_0=1 + 2(\log(\theta) +\log(1-\theta))$.
\\Let $h(x)=x^p\log(x)$. Using $\mathcal{U}_{\uu}\mathcal{B}f=\alpha \mathcal{U}_{\uu}f$, we obtain
\begin{equation*}
\alpha l_p = \alpha (pl_p +m_p) + m_{p+1}(\theta^p \log(\theta) +(1-\theta)^p\log(1-\theta))+ l_{p+1}(\theta^p + (1-\theta)^p -1)
\end{equation*} 
that allows us to conclude the proof.
\fi
\end{proof}

%\subsubsection{Proof of Theorem \ref{th:inegpartial}}
%
%Using Lemma~\ref{lem:Utogamma}, we are able to prove Theorem \ref{th:inegpartial}. 

We now end the paper with the proof of Theorem~\ref{th:inegpartial}.

\begin{proof}[Proof of Theorem~\ref{th:inegpartial}]
Let $\uu=(\alpha,0,\theta)$, then we have $\lambda(\uu)=\alpha$ and $h_{\uu}(x,m)=x$. By Theorem~\ref{pr:variationmalthus}~$(ii)$, we obtain

\begin{equation*}
\frac{\partial \lambda(\uu)}{\partial \epsilon}=\frac{\partial \lambda(\alpha,0,\theta)}{\partial \epsilon}=\int_{0}^{+\infty}x(\gamma_{\uu}^1(x) -\gamma_{\uu}^0(x))dx.
\end{equation*}
By Lemma~\ref{lem:Utogamma}, we have
\begin{align*}
\int_{0}^{+\infty}x\gamma_{\uu}^1(x)dx &=\int_{0}^{+\infty}\frac{e^{-\frac{x-1}{\alpha}}}{x}\int_{0}^{x/(1-\theta)}(1-\theta)ze^{\frac{(1-\theta)z-1}{\alpha}}B(z)\mathcal{U}_{\uu}(z)dzdx\\
&=\int_{0}^{+\infty}\int_{z(1-\theta)}^{+\infty}\frac{e^{-\frac{x-1}{\alpha}}}{x}(1-\theta)z e^{\frac{(1-\theta)z-1}{\alpha}}B(z)\mathcal{U}_{\uu}(z)dxdz\\
&=\int_{0}^{+\infty}G(z(1-\theta))B(z)\mathcal{U}_{\uu}(z)dz,
\end{align*}
where 
\begin{equation*}
G(t)=t e^{\frac{t-1}{\alpha}}\int_{t}^{+\infty}\frac{e^{-\frac{x-1}{\alpha}}}{x}dx.
\end{equation*}
So we obtain
\begin{equation}\label{eq:inegmoment}
\int_{0}^{+\infty}x(\gamma_{\uu}^1(x)-\gamma_{\uu}^0(x))dx=\int_{0}^{+\infty}(G(z(1-\theta))-G(z\theta))B(z)\mathcal{U}_{\uu}(z)dz.
\end{equation}
We conclude by studying the monotonicity of $G$. We rewrite 
\begin{equation*}
G(t)=t e^{\frac{t}{\alpha}}\int_{t}^{+\infty}\frac{e^{-\frac{x}{\alpha}}}{x}dx,
\end{equation*}
and we obtain 
\begin{equation*}
G'(t)=\left(1+\frac{t}{\alpha}\right)e^{\frac{t}{\alpha}}\int_{t}^{+\infty}\frac{e^{-\frac{x}{\alpha}}}{x}dx -1.
\end{equation*}
We apply Jensen inequality to the strictly convex function $x\mapsto 1/x$ and the measure $(\alpha e^{-t/\alpha})^{-1} 1_{\left[t,+\infty\right)}(x)e^{-\frac{x}{\alpha}}dx$, that allows to obtain
\begin{align*}
(\alpha e^{-\frac{t}{\alpha}})^{-1}\int_{t}^{+\infty}\frac{e^{-\frac{x}{\alpha}}}{x}dx &> \frac{1}{(\alpha e^{-\frac{t}{\alpha}})^{-1}\int_{t}^{+\infty}x e^{-\frac{x}{\alpha}}}\\
& =\frac{1}{(\alpha e^{-\frac{t}{\alpha}})^{-1}(\alpha e^{-\frac{t}{\alpha}}+\alpha^2 e^{-\frac{t}{\alpha}})}\\
&=\frac{1}{t+\alpha}.
\end{align*}
It comes that
\begin{equation*}
\int_{t}^{+\infty}\frac{e^{-\frac{x}{\alpha}}}{x}dx > \frac{\alpha e^{-\frac{t}{\alpha}}}{t+\alpha},
\end{equation*}
and that $G'(t) > 0$. Equation (\ref{eq:inegmoment}) allows us to conclude the proof.
\end{proof}

\begin{rema}%\cmbc{cf mon brouillon mes inv p 6}
Mimicking the proof of Theorem~\ref{th:inegpartial}, we can show that, for general division rate $B$, if
$$
G:u \mapsto ue^{\int_1^u \frac{B(r)}{\alpha r} dr} \int_u^{+\infty} \frac{e^{-\int_1^x\frac{B(r)}{\alpha r}}}{x} dx
$$
is increasing (which seem at least right for $x^p$ for $p\leq 1$) then the conclusion of Theorem~\ref{th:inegpartial} also holds true. %
%If $G$ is decreasing then if $1-\theta < \theta$ then $\frac{\partial \lambda}{\partial \epsilon} > 0.$ However $G(0)=0$ and $G(u)>0$ for $u>0$ then this case seems impossible. \cmbc{a modifier j'ai ecrit ca en partie pour nous}
\end{rema}

\paragraph{Acknowledgements.} This work was partially suppported by the Chaire Modélisation Mathématique et Biodiversité of Veolia Environment - \'Ecole Polytechnique - Museum National d'Histoire Naturelle - FX, and the ANR project MESA (ANR-18-CE40-006), funded by the French Ministry of Research.

\end{document}